\documentclass[11pt,leqno,twoside]{article}
\usepackage[english]{babel}
\usepackage{mathrsfs,amssymb,amsfonts,amsthm,amsmath,natbib,amsbsy}
\usepackage{dsfont,verbatim,fancyhdr}
\usepackage{graphicx,float}
\usepackage[hang]{caption}
\usepackage[usenames,dvipsnames]{xcolor}
\usepackage[bookmarks,bookmarksnumbered,colorlinks=true,pdfstartview=FitV, linkcolor=Red,citecolor=blue!90!green,urlcolor=blue!90!green]{hyperref}
\usepackage{tikz}
\usepgflibrary{arrows}
\usepackage[normalem]{ulem}
\usepackage{marginnote,enumerate,enumitem}
\usepackage{pdfsync}

\allowdisplaybreaks
\bibpunct{\textcolor{blue!90!green}{(}}{\textcolor{blue!90!green}{)}}{\textcolor{blue!90!green}{;}}{a}{\textcolor{blue!90!green}{,}}{\textcolor{blue!90!green}{;}}

\def\d{\mathrm{d}}
\def\X{\mathcal{X}}%: the state-space of hmm signal
\def \Y{\mathcal{Y}} %: the state-space of hmm data
\def\M{\mathcal{M}} %: the state-space of the death process
\def\MM{\mathscr{M}} %: space of measures
\def\Z{\mathbb{Z}_+}%:the integers with 0
\def\r{r} % the drift of dynamical system

\def\N{\mathbb{N}}
\def \R{\mathbb{R}}
\def \E{\mathbb{E}}
\def\F{\mathcal{F}}
\def\L{\mathcal{L}} %law of random variables
\def\Gfv{\mathbb{A}} %generator of FV
\def\Gdw{\mathbb{B}} %generator of DW
\def\Gwf{\mathcal{A}_{K}} %generator of WF
\def\Gcir{\mathcal{B}} %generator of CIR
\def\Gn{B} %generator of CIR dual
\def\A{\mathcal{A}}
\def\Fdir{\mathscr{F}_{\Pi}}
\def\Fgamma{\mathscr{F}_{\Gamma}}

\def\DK{\Delta_{K}}
 
\def\ci{\perp\!\!\!\perp}

\def \la {\langle}
\def \ra {\rangle}

\newcommand{\bs}[1]{\boldsymbol{#1}}

\newcommand{\norm}[1]{|{#1}|}
\def\na{\theta}

\def \aa {\bs\alpha}
\def \yy {\mathbf{y}}

\def \ee {\mathbf{e}}

\def \mm {\mathbf{m}}
\def \nn {\mathbf{n}}
\def \ii {\mathbf{i}}
\def \oo {\mathbf{0}}

\def \xx {\mathbf{x}}
\def \zz {\mathbf{z}}
\def \XX {\mathbf{X}}

\def \ZZ {\mathbf{Z}}
\def \ss {s} %: sample size
\def \SS {S} %: sample size

\def\hw{\widehat{w}}

\newtheorem{theorem1}{Theorem}[section]
\newenvironment{theorem}{\begin{theorem1}}{\end{theorem1}}
\newtheorem{proposition}[theorem1]{Proposition}
\newtheorem{definition1}[theorem1]{Definition}

\newtheorem{remark1}[theorem1]{Remark}

\newtheorem{example1}[theorem1]{Example}

\newtheorem{lemma1}[theorem1]{Lemma}
\newenvironment{lemma}{\begin{lemma1}}{\end{lemma1}}

\newcounter{algorithm}

\renewcommand{\thealgorithm}{\Alph{algorithm}}

\usepackage{ifthen}
\newcommand{\todo}[1]{ \ifthenelse{ \equal{\draft}{true} }{{\color{red} #1 }}{}}
\newcommand{\draft}{true}

\long\def\symbolfootnote[#1]#2{\begingroup\def\thefootnote{\hspace*{-1mm}\fnsymbol{footnote}}\footnote[#1]{#2}\endgroup}

\fancypagestyle{plain}{
\fancyhead{}
}

\title{\bf \vspace{-2.5cm}
Filtering hidden Markov measures \\[5mm]
}
\author{
\normalsize\textsc{Omiros Papaspiliopoulos}\\[1mm]
\normalsize\emph{ICREA and Universitat Pompeu Fabra}\\[2mm]
\normalsize\textsc{Matteo Ruggiero%\footnote{Email: matteo.ruggiero@unito.it}
}\\[1mm]
\normalsize\emph{University of Torino and Collegio Carlo Alberto}\\[2mm]
\normalsize\textsc{Dario Span\`o}\\[1mm]
\normalsize\emph{University of Warwick}
}
\date{\today}

\begin{document}

\maketitle
\thispagestyle{empty}

\vspace{-5mm}
\begin{center}
\begin{minipage}{.75\textwidth}
\footnotesize\noindent
We consider the problem of learning two families of time-evolving
random measures from
indirect observations. In the first model, the signal
is  a Fleming--Viot diffusion, which is reversible with respect to the law of a Dirichlet process, and the data is a sequence of random samples
from the state at discrete times. In the second model, the
signal is a  Dawson--Watanabe diffusion, which is reversible with
respect to the law of a gamma random measure, and the data is a sequence of Poisson point
configurations whose intensity is given by the state at discrete
times. A common methodology is developed to obtain the filtering
distributions  in a computable
form, which is based on the projective
properties of the signals and duality properties of their
projections. The filtering distributions take the form of mixtures of
Dirichlet processes and gamma random measures for each of the two
families respectively, and an explicit algorithm is provided to
compute the parameters of the mixtures. Hence, our results extend
classic characterisations of the posterior distribution under
Dirichlet process and gamma random measures priors to
a dynamic framework. \\[-2mm]


\end{minipage}
\end{center}
%\renewcommand{\baselinestretch}{.5}
%\setlength{\parskip}{1mm}

%\linespread{.4}
%\tableofcontents%\newpage	
\linespread{1.2}

%%%%%%%%%%%%%%%%%%%%%%%%%%%%%%%%

\section{Introduction}%\label{sec: intro}

\subsection{Hidden Markov measures}%\label{}

A hidden Markov model (HMM) is a  sequence
$\{(X_{t_{n}},Y_{t_{n}}),n\ge1\}$, with the following ingredients:
$X_{t_{n}}$ is an unobserved Markov chain, called \emph{latent signal}
and assumed here to be the discrete time sampling of a continuous time
Markov process; the $Y_{t_{n}}$'s are conditionally independent
observations given the signal, with law given by the \emph{emission
  distribution}  $\L(Y_{t_{n}}| X_{t_{n}})$, parametrised by the
current signal state. Filtering optimally an HMM entails the sequential exact evaluation of the \emph{filtering distributions} $\L(X_{t_{n}}|
Y_{1:n})$, that is the conditional distributions of
the signal given the past and current observations $Y_{1:n}:=(Y_{t_{1}},\ldots,Y_{t_{n}})$. Optimal filtering thus extends the Bayesian approach to a dynamic framework. The evaluation of the filtering  distributions is the key for the solution of several statistical problems in this setting, such as the prediction of future observations, the derivation of smoothing distributions and the calculation of likelihood functions. See \cite{CMR05} for a book-length treatment of the recursions involved in such computations and their dependence on the filtering distributions.

The literature on HMMs has largely focussed on parametric signals,
where the unobserved Markov process is finite dimensional.
Recently,  authors have
considered infinite dimensional HMMs, which involve the dynamics of
infinitely many parameters.
 One strand of work originates with
\cite{BGR02}, who model the
signal as a Markov chain with countable state space and transitions
based on a hierarchy of Dirichlet processes. See also \cite{VSTG08},
\cite{SGGL09} and \cite{ZZZ14} for further developments.  A different
strand of work tries to build time-evolving Dirichlet processes for
semi-parametric time-series analysis; see for example \cite{GS06}, \cite{RT08} and \cite{MR15}, which all build upon the celebrated stick breaking representation of the Dirichlet process \citep{S94}. Finally, yet another class of
infinite dimensional HMMs takes
the signal as a Markov chain with finite state space but uses an
infinite number of parameters for the emission distribution, see
\cite{YPRH11}. A common feature of the above mentioned contributions is that they all resort to Monte Carlo strategies for posterior computation.

In this paper we study models for time-evolving
measures and derive their posterior distributions analytically. We consider two families of models that give rise to
infinite dimensional or measure-valued hidden Markov models, and term
these families \emph{hidden Markov measures}. In the first family we consider two models. The simpler of the two assumes that the
signal at each time point is a probability distribution on a countable
set, with infinitely many parameters. The signal evolves in continuous time according an infinite Wright--Fisher diffusion, and the data are available in discrete times and are realisations from the distribution given by the signal. In the more general model, the signal at each time point is a discrete distribution on a Polish space, evolving as a
Fleming--Viot diffusion, and the data are obtained in discrete times as
draws from the underlying measure. The former model admits a likelihood and can be dealt with by means of direct methods, whereas the latter features an evolving support for the signal states and needs to be manipulated indirectly, hence they are treated separately. 
 In the second family the signal at each time point is a positive almost surely discrete measure on a Polish space, the signal evolves in continuous time according to the Dawson--Watanabe diffusion, and the data are available in discrete times as realisations from a doubly stochastic Poisson process with intensity given by the signal.

In our specification, the Fleming--Viot and Dawson--Watanabe processes are stationary with respect to the laws of Dirichlet and gamma random measures respectively. Additionally, the processes can be defined so that one
parameter controls the correlation structure and other parameters
determine the invariant distribution. Therefore, our models are
natural \emph{dynamic} extensions of infinite dimensional
\emph{static} models for unknown distributions and intensities that
are widely used in Bayesian statistics and machine learning for a
broad range of applications. A fundamental reason for the popularity
of the static models are \emph{conjugacy} properties that make the
Bayesian updating tractable.

In this paper we demonstrate that Bayesian learning is tractable for
the families of hidden Markov measures we consider. We show that the
filtering distributions evolve within finite mixtures of Dirichlet and gamma
random measures, and we provide a recursive algorithm for the
computation of the parameters of these mixtures. Broadly speaking, our
theory builds upon a synthesis of three classes of background
results. The first is the connection between
filtering and the so-called dual process, which was recently
established in \cite{PR14} and is reviewed in Section \ref{sec: filtering and duality}. This previous work identifies classes of parametric HMMs for which the filtering distributions evolve in finite
mixtures of finite dimensional distributions and provides a recursive
algorithm for the associated parameters. The second class of results is concerned with the projective properties of the Fleming--Viot and Dawson--Watanabe processes that
link them to the finite dimensional Wright--Fisher and Cox--Ingersoll--Ross processes;
the details are discussed in Sections \ref{subsec: FV-model} and
\ref{subsec: DW-model} respectively. The third relates to the conjugacy properties of the
corresponding  static models and mixtures thereof, discussed in
Sections \ref{subsec: FV-static} and \ref{subsec: DW-static}. Our
strategy exploits the fact that the finite dimensional projected
models admit computable filters, due to the results in \cite{PR14},
and that the exchange of the operations ``projection'' and ``filter''
is valid, which we prove in this paper. Figure \ref{scheme} depicts the strategy for obtaining our results. Given the signal distribution $\L(X_{t_{n}}\mid Y_{1:n})$, its time propagation $\L(X_{t_{n+k}}\mid Y_{1:n})$ is found by propagating in time the projection of the former onto an arbitrary partition, that is $\L(X_{t_{n}}(A_{1},\ldots,A_{K})\mid Y_{1:n})$, and by exploiting the projective characterisation of the filtering distributions. This is done in Theorems \ref{prop: FV propagation} and \ref{prop: DW propagation} for the Fleming--Viot and Dawson--Watanabe case respectively, while the same result is obtained in Theorem \ref{prop: PR14-WF-infty} for the infinite alleles model by  exploiting directly the duality properties.
In addition, a duality result and the time propagation are provided in Theorem \ref{multi CIR duality} and Proposition \ref{K-CIR implications} for a class of multivariate Cox--Ingersoll--Ross processes.

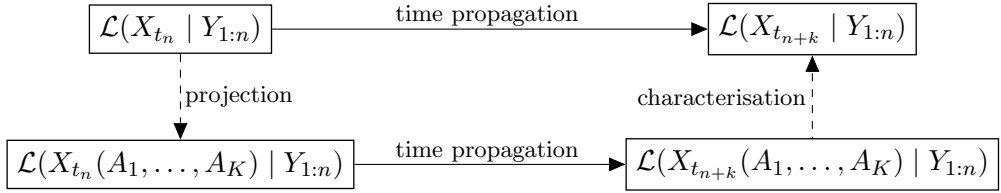
\begin{figure}[t!]
\begin{center}
\begin{tikzpicture}[>=triangle 45,scale=.6]
\node (fv1) at (-2,3) [rectangle,draw,line width=.5pt]
	{$\L(X_{t_{n}}\mid Y_{1:n})$};
\node (fv2) at (12,3) [rectangle,draw,line width=.5pt]
	{$\L(X_{t_{n+k}}\mid Y_{1:n})$};
\node (wf1) at (-2,0) [rectangle,draw,line width=.5pt]
	{$\L(X_{t_{n}}(A_{1},\ldots,A_{K})\mid Y_{1:n})$};
\node (wf2) at (12,0) [rectangle,draw,line width=.5pt]
	{$\L(X_{t_{n+k}}(A_{1},\ldots,A_{K})\mid Y_{1:n})$};
%	\node (conj1) at (12,0) [rectangle,draw,line width=.5pt]
%	{$\mu_{\mm,\yy^{*}_{m}}$};
%\node (conj2) at (12,-3) [rectangle,draw,line width=.5pt]
%	{$\mu_{\tilde\mm}$};
\draw[->] (fv1) -- (fv2);
\node at (4.8,3.3) {\footnotesize\text{time propagation}};
\draw[->,dashed] (fv1) -- (wf1);
\node at (-0.7,1.5) {\footnotesize\text{projection}};
\draw[->] (wf1) -- (wf2);
%\node at (5.5,-2.7) {\text{time}};
\draw[<-,dashed] (fv2) -- (wf2);
\node at (10,1.5) {\footnotesize\text{characterisation}};
%\draw[->,dashed] (conj1) -- (conj2);
\node at (4.8,0.3) {\footnotesize\text{time propagation}};
%\draw[double, thick] (wf2) -- (conj2);
%\node at (10.1,-2.7) {\footnotesize\text{d}};
\end{tikzpicture}
\begin{minipage}{.75\textwidth}\vspace{4mm}
\caption{\footnotesize Given the same amount of data, the time propagation of the filtering distribution $\L(X_{t_{n}}\mid Y_{1:n})$ is found by propagating its projection onto an arbitrary partition $(A_{1},\ldots,A_{K})$, and by exploiting the projective characterisation of the filtering distributions.}\label{scheme}
\end{minipage}
\end{center}
\end{figure}

%\red{Finally, we also show in Section \ref{sec:ima} that a computable filter can be derived via duality for the so-called infinitely-many-neutral allele diffusion of Population Genetics, describing the evolution of the unlabeled frequencies of a FV process, despite the stationary distribution of the signal, the Poisson-Dirichlet distribution, is not conjugate with respect to \emph{iid} sampling. Indeed we show that a loosened form of conjugacy property (namely, a representation as finite mixture of conjugate distributions), holding for the Poisson-Dirichlet distribution, is sufficient to guarantee and the filter's computability.}

\subsection{Notations}

Throughout the paper, $\Y$ will denote a locally compact Polish space, $\MM(\Y)$ the space of finite Borel measures on $\Y$, $\MM_{1}(\Y)$ its subspace of probability measures. A typical element of $\MM(\Y)$ will be
\begin{equation}\label{alpha measure}
\alpha\in\MM(\Y), \quad
\alpha=\theta P_{0}, \quad
\theta>0,\quad
P_{0}\in\MM_{1}(\Y),
\end{equation}
where $\theta$ is the total mass of $\alpha$. The discrete measures $x(\cdot)\in\MM_{1}(\Y)$ and $z(\cdot)\in\MM(\Y)$ will denote the marginal states of the signals $X_{t}$ and $Z_{t}$, with $X_{t}(A)$ and $Z_{t}(A)$ being the respective one dimensional projections onto the Borel set $A\subset\Y$. We will also adopt
boldface notation to denote vectors, with the following conventions:
\begin{equation*}\label{vector notation}
\begin{split}
\xx=&\,(x_{1},\ldots,x_{K})\in\R_{+}^{K}, \quad
\mm=\,(m_{1},\ldots,m_{K})\in \Z^{K}, \quad\ \
\xx^{\mm}=x_{1}^{m_{1}}\cdots x_{K}^{m_{K}},\quad
|\xx|=\sum_{i=1}^{K}x_{i},
\end{split}
\end{equation*}
where the dimension $2\le K\le \infty$ will be clear from the context if not specified explicitly. Typically, $\xx$ will represent a finite dimensional signal state and $\mm$ a vector of multiplicities.

%%%%%%%%%%%%%%%%%%%%%%%%%

\section{Computable filtering and duality}\label{sec: filtering and duality}

Computable filtering refers to the circumstance where the filtering distributions can be characterised
by a finite number of parameters whose computation can be achieved at
cost that grows at most polynomially with the number of
observations. Special cases of this framework are finite dimensional
filters for which the computational cost is linear in the number of
observations, the Kalman filter for linear Gaussian HMMs being the
celebrated model in this setting.

\cite{PR14} recently developed a framework for the computable
filtering of finite dimensional HMMs.
A brief description of the framework is as follows. There is a finite
dimensional HMM $\{(X_{t_{n}},Y_{t_{n}}),n\ge0\}$, where $X_{t}$ has
state space $\X$,
stationary distribution $\pi$, transition kernel $P_t(x,\d x')$, and the data are linked to the signal
via an emission density $f_{x}(y)$. The filtering distributions are
$\nu_n:=\L(X_{t_n} | Y_{1:n})$ and $\nu$ is
the prior distribution for $X_{t_0}$.
The exact or optimal filter is the solution of the recursion
\begin{equation*}%\label{}
\nu_0 = \phi_{Y_{t_{0}}} (\nu)\,,\quad \nu_{n} = \phi_{Y_{t_{n}}}(\psi_{ t_n - t_{n-1}}(\nu_{n-1})),\quad\,n\in\N,
\end{equation*}
which involves the following two operators acting on measures: the \emph{update operator}
\begin{equation}\label{update operator}
\phi_{y}(\nu)(\d x)
=\frac{f_x(y) \nu(\d x)}{p_\nu(y)}, \quad \quad p_\nu(y) = \int_{\X} f_x(y)
\nu(\d x)\,,
\end{equation}
 and the \emph{prediction operator}
\begin{equation}\label{prediction operator}
\psi_t(\nu)(\d x')%=\nu P_t(\d x')
=\int_{\X}\nu(\d x)P_t(x,\d x').
\end{equation}
The existence of a computable filter and a recursive algorithm can be
established if the following structure is embedded in the HMM:
\begin{itemize}
\item \emph{Conjugacy}: there exists a function $h(x,\mm,\theta)\geq 0$,
  where $x \in \X$,  $\mm \in \Z^{K}$ for some
  $K\in\N$, and $\theta \in \R^l$ for some $l\in\N$, and functions $t(y,\mm)$
  and $T(y,\theta)$ such that $\int h(x,\mm,\theta) \pi(dx)  =1$, for all $\mm$ and $\theta$, and
\[
\phi_y (h(x,\mm,\theta) \pi(dx) ) = h(x,t(y,\mm),T(y,\theta))\pi(dx).
\]
\item \emph{Duality}: there exists a two-component Markov
process $(M_t,\Theta_t)$ with state-space $\Z^{K} \times \R^l$, and
generator,
 \begin{equation*}%\label{}
\label{eq:gen-dual}
  (A g)(\mm,\theta) = \lambda(|\mm|) \rho(\theta) \sum_{i=1}^K m_i
   [g(\mm-\ee_i,\theta)-g(\mm,\theta)] + \sum_{i=1}^{l}\r_{i}(\theta)   \frac{\partial g(\mm,\theta)}{\partial \theta}\,,
\end{equation*}
such that it is \emph{dual} to $X_t$ with respect to the function $h$,
i.e., it satisfies
\begin{equation}
  \label{duality identity}
  \E^x[h(X_t,\mm,\theta)] = \E^{(\mm,\theta)}[h(x,M_t,\Theta_t)],  \quad
  \forall x \in \X, \mm \in \Z^K, \theta \in \R^l, t\geq 0\,.
\end{equation}
\end{itemize}
Note that the $\Theta$  component of the dual process is assumed to
evolve autonomously, according to a system of ordinary differential
equations, and modulates the rates of $M_t$, which is a death
process on a $K$-dimensional lattice. Under these conditions,
Proposition 2.3 of \cite{PR14} shows that for $\F=\{h(x,\mm,\theta)
\pi(\d x),\, \mm \in \Z^K, \theta \in \R^l\}$, if $\nu \in \F$, then
$\nu_n$ is a finite mixture of distributions in $\F$ with parameters
that can be computed recursively. Additionally, the local
sufficient condition for duality
\begin{equation}\label{local duality}
(\A h(\cdot,\mm,\theta))(x)
=(A h(x,\cdot,\cdot))(\mm,\theta), \quad \quad \forall x\in \X, \mm
\in \Z^K, \theta \in \R^l,
\end{equation}
where $\A$ denotes the generator of $X_t$, is applied to identify dual processes and computable filters when the signal is the Cox--Ingersoll--Ross or the $K$-dimensional Wright--Fisher diffusion.
The work of \cite{PR14} includes as special cases computable filters
obtained previously in \cite{GK04} and \cite{CG06,CG09}. However, it is strictly
applicable to finite dimensional signals, due to the assumptions that
an emission density and a finite dimensional dual exist.

%%%%%%%%%%%%%%%%%%%%%%%%%

%%%%%%%%%%%%%%%%%%%%%%%%%

\section{Filtering Fleming--Viot processes}\label{sec: FV}

\subsection{The static model: Dirichlet process}\label{subsec: FV-static}

The Dirichlet process, introduced by \cite{F73} and commonly
recognised as the cornerstone in Bayesian nonparametrics (see \cite{G10} for a recent review), is a
discrete random probability measure $x\in\MM_{1}(\Y)$ that can be
thought to describe the frequencies in a population with
infinitely many labelled types, whereby $\Y$ is often referred to as the
 \emph{type space}.
The process admits the series representation
\begin{equation}\label{DP series representation}
x(\cdot)=\sum_{i=1}^{\infty}W_{i}\delta_{Y_{i}}(\cdot), \quad
W_{i}=\frac{Q_{i}}{\sum_{j\ge1}Q_{j}},\quad
Y_{i}\overset{iid}{\sim}P_{0},\quad (Y_i)\ci (W_i).
\end{equation}
Here $\{Q_{i},i\ge1\}$ are the jumps of a gamma process
with mean measure $\na y^{-1}e^{-y}\d y$. %, so that the ranked weights $\{W_{(i)},i\ge1\}$, independently of the locations, have the Poisson--Dirichlet distribution with parameter $\na$ \citep{K75}. 
We will denote by $\Pi_{\alpha}$, $\alpha=\theta P_{0}$, the law of  $x(\cdot)$ in \eqref{DP series representation}.
The Dirichlet process has two fundamental properties that are of great
interest in statistical learning:
\begin{itemize}
\item \emph{Conjugacy}: $y_{i}\mid x\overset{iid}{\sim}x$ and $x\sim\Pi_{\alpha}$ imply $x\mid \yy_{1:m}\sim
\Pi_{\alpha+\sum_{i=1}^{m}\delta_{y_{i}}}$, where $\yy_{1:m}:=(y_{1},\ldots,y_{m})$.
\item \emph{Projection}: for a measurable partition $A_{1},\ldots,A_{K}$ of
  $\Y$, the vector $(x(A_{1}),\ldots,x(A_{K}))$ has the Dirichlet
  distribution with parameter
  $\aa=(\alpha(A_{1}),\ldots,\alpha(A_{K}))$, henceforth denoted
  $\pi_{\aa}$.
\end{itemize}
When $\Y=\N$, the Dirichlet process can be seen as a Dirichlet distribution with infinitely many types defined on
\begin{equation}\label{infinite simplex}
\Delta_{\infty}=\left\{x\in[0,1]^{\infty}:\ \sum_{i=1}^{\infty}x_{i}= 1\right\}.
\end{equation}
This corresponds to the construction
\begin{equation}\label{infinite-dimensional dirichlet}
V_{i}=\frac{T_{i}}{\sum_{j\ge1}T_{j}}, \quad \quad
T_{j}\overset{ind}{\sim}\text{Ga}(\alpha_{j},1),
\end{equation}
where
\begin{equation}\label{inf-alpha constraints}
\aa=(\alpha_{1},\alpha_{2},\ldots), \quad \quad \alpha_{i}=\theta P_{0}(\{i\}), \quad i\ge1.
\end{equation}
See \cite{EG93}, equation (1.26) and Lemma 2.2. 
With a slight abuse of notation, we will denote the law of $(V_{1},V_{2},\ldots)$ by the same symbol $\pi_{\aa}$ used for the Dirichlet distribution; which distribution $\pi_{\aa}$ refers to will be clear from the context. 

Mixtures of Dirichlet processes were introduced in \cite{A74}. They add  a further level to the Bayesian hierarchical model, whereby
\begin{equation*}%\label{}
y_{i}\mid x,u\overset{iid}{\sim}x,\quad \quad x\mid u\sim\Pi_{\alpha_{u}},\quad \quad
u\sim H,
\end{equation*}
where $\alpha_{u}$ denotes the measure $\alpha$ conditionally on $u$. Equivalently,
\begin{equation*}%\label{}
y_{i}\mid x\overset{iid}{\sim}x,\quad \quad
x\sim \int_{\mathcal{U}}\Pi_{\alpha_{u}}\d H(u).
\end{equation*}
For mixtures of Dirichlet processes the conjugacy and projective properties read as follows:
\begin{itemize}
\item \emph{Conjugacy}:
$y_{i}\mid x\overset{iid}{\sim}x$ and $x\sim \int_{\mathcal{U}}\Pi_{\alpha_{u}}\d H(u)$
imply
\begin{equation*}%\label{}
x\mid \yy_{1:m}
\sim \int_{\mathcal{U}}\Pi_{\alpha_{u}+\sum_{i=1}^{m}\delta_{y_{i}}}\d H_{\yy_{1:m}}(u),
\end{equation*}
where $H_{\yy_{1:m}}$ is the conditional distribution of $u$
given $\yy_{1:m}$.

\item \emph{Projection}: for a measurable partition $A_{1},\ldots,A_{K}$ of
  $\Y$, we have
  \begin{equation*}%\label{}
(x(A_{1}),\ldots,x(A_{K}))
\sim
\int_{\mathcal{U}}\pi_{\aa_{u}}\d H(u),
\end{equation*}
where $\aa_{u}=(\alpha_{u}(A_{1}),\ldots,\alpha_{u}(A_{K}))$.
\end{itemize}
%Such mixtures are also conjugate to random sampling, upon appropriate update of the base measure $\alpha_{u}$ and of the mixing distribution $H$.
%, with density with respect to the Lebesgue measure proportional to $\xx^{\aa-\one}=x_{1}^{\alpha_{1}-1}\cdots x_{K}^{\alpha_{K}-1}$
%%\begin{equation}\label{dirichlet density}
%%\pi_{\aa}=\frac{\Gamma(\na)}{\prod_{i=1}^{K}\Gamma(\alpha_{i})}
%%\xx^{\aa-\one}\d \xx\,\ind_{\DK}(\xx),
%%\end{equation}
%on the $(K-1)$-dimensional simplex
%\begin{equation}\label{eq:K-simplex}\notag
%\DK=\Big\{x\in[0,1]^{K}:\, \sum_{i=1}^{K}x_{i}=1\Big\}.
%\end{equation}
%%Since we are primarily interested in parameter changes, we will often write $\pi_{\aa}$ in place of $\pi_{ \aa}$ when this causes no confusion.

%%%%%%%%%%%%%%%%%%%%%%%%%

\subsection{The signal: Fleming--Viot process}\label{subsec: FV-model}

Fleming--Viot (FV) processes are a family of diffusions taking values
in the subspace of $\MM_{1}(\Y)$ given by purely atomic probability
measures, hence they describe evolving discrete distributions.  See \cite{EK93} and \cite{D93} for exhaustive reviews. Here we restrict the attention to a subclass known as the (labelled) \emph{infinitely many neutral alleles model} with parent independent mutation, henceforth for simplicity called the FV process.
This is characterised by the generator
\begin{align}\label{FV-generator}
\Gfv \varphi(x)
=&\,\frac{1}{2}\int_{\Y}\int_{\Y}x(\d y)(\delta_{y}(\d u)-x(\d u))\frac{\partial^{2}\varphi(x)}{\partial x(\d y)\partial x(\d u)}
+\int_{\Y}x(\d y)B\bigg(\frac{\partial\varphi(x)}{\partial x(\cdot)}\bigg)(y)
%=&\,\sum_{i=1}^{m}\la M_{i}f,x^{m}\ra
%+\frac{1}{2}\sum_{\substack{1\le k\ne i\le m}}\la \Phi_{ki}f-f,x^{m}\ra,
\end{align}
where $\varphi$ is a test function, $\delta_{y}$ is a point mass at $y$ and $\partial \varphi(x)/\partial x(\d
y)=\lim_{\varepsilon\rightarrow0^{+}}\varepsilon^{-1}(\varphi(x+\varepsilon
\delta_{y})-\varphi(x))$. Furthermore, $B$ is the  \emph{mutation operator},
that is the generator of a Feller pure-jump process on $\Y$
\begin{equation}\label{mutation operator}
(Bf)(y)=\frac{\na}{2}\int_{\Y}[f(u)-f(y)]P_{0}(\d u),\quad \quad f\in C(\Y),
\end{equation}
whereby at rate $\na/2$ jumps occur to a location sampled from
$P_{0}\in\MM_{1}(\Y)$ independently of the current value of the jump
process, whence mutations are \emph{parent independent}. The domain of $\Gfv$ is taken to be the set of $\varphi\in C(\MM_{1}(\Y))$ of the form $\varphi(x)=F(\la f_{1},x\ra,\ldots,\la f_{k},x\ra)$, where $f_{i}\in C(\Y)$ and $F\in C^{2}(\R^{k})$.
This FV process is known to be stationary and reversible with respect to the law $\Pi_{\alpha}$ of a Dirichlet process as in \eqref{DP series representation}; see \cite{EK93}, Section 8.

Projecting a FV process $X_{t}$ onto a measurable partition $A_{1},\ldots,A_{K}$ of $\Y$ yields a $K$-dimensional Wright--Fisher (WF) diffusion, which is reversible and stationary with respect to the Dirichlet distribution $\pi_{\aa}$, for $\alpha_{i}=\theta P_{0}(A_{i})$, $i=1,\ldots,K$, and has infinitesimal operator,
for $x_{i}=x(A_{i})$,
\begin{equation}\label{eq:WF-generator}
\Gwf f(\xx)=
\frac{1}{2}\sum_{i,j=1}^{K}x_{i}(\delta_{ij}-x_{j})\frac{\partial^{2}f(\xx)}{\partial x_{i}\partial x_{j}}
+\frac{1}{2}\sum_{i=1}^{K}(\alpha_{i}-\na x_{i})\frac{\partial f(\xx)}{\partial x_{i}}.
\end{equation}
Here $\delta_{ij}$ denotes Kronecker delta and $\Gwf$ acts on
$C^{2}(\DK)$ functions. See \cite{D10}. This property is the dynamic
counterpart of the projective property of Dirichlet processes
discussed earlier. The same result is obtained when the mutant type distribution $P_{0}$ is finitely supported, for example when $\Y=\{1,\ldots,K\}$.

%It has been shown that a duality relation, in the sense of \eqref{local duality}, holds between the WF diffusion with generator \eqref{eq:WF-generator} and a $K$-dimensional death process $\{M_{t},t\ge0\}$ on $\Z^{K}$, with generator
%\begin{equation}\label{WF dual gen}
%\Gm g(\mm)
%=\lambda_{|\mm|}\sum_{i=1}^{K}\frac{m_{i}}{|\mm|}\big(g(\mm-\ee_{i})-g(\mm)\big),
%\quad \quad \lambda_{|\mm|}=|\mm|(\na+|\mm|-1)/2,
%\end{equation}
% $\mm \in \Z^{K}$, $g$ bounded on $\Z^{K}$. The duality is with respect to the function
%\begin{equation}\label{WF duality function}
%h(\xx,\mm)=\frac{\Gamma(\na+|\mm|)}{\Gamma(\na)}\prod_{i=1}^{K}\frac{\Gamma(\alpha_{i})}{\Gamma(\alpha_{i}+m_{i})}\xx^{\mm}.
%\end{equation}
%See, \cite{BEG00}, \cite{EG09} and \cite{PR14}.

When $K$ goes to infinity in \eqref{eq:WF-generator}, or alternatively when $\Y=\N$ in \eqref{FV-generator}, the signal follows a diffusion characterised in \cite{E81}, with generator
\begin{equation}\label{eq:ethier-generator}
\A_{\infty}f(\xx)=
\frac{1}{2}\sum_{i,j=1}^{\infty}x_{i}(\delta_{ij}-x_{j})\frac{\partial^{2} f(\xx)}{\partial x_{i}\partial x_{j}}
-\frac{1}{2}\sum_{i=1}^{\infty}(\alpha_{i}-\theta x_{i})\frac{\partial f(\xx)}{\partial x_{i}},
\end{equation}
where $x_{i}=X(\{i\})$ and $(\alpha_{1},\alpha_{2},\ldots)$ is as in \eqref{inf-alpha constraints}.
Here $\A_{\infty}$ acts on functions $C^{2}(\overline\Delta_{\infty})$
which depend on finitely many coordinates,
$\overline\Delta_{\infty}$ denoting the closure of \eqref{infinite simplex} in the product topology,
and the associated process is reversible with respect to the law $\pi_{\aa}$ of
\eqref{infinite-dimensional dirichlet}. See also \cite{EG93} and
\cite{EK93}. In this paper will refer to this process as the
\emph{infinite WF diffusion}, given its structural similarity with \eqref{eq:WF-generator}.

For statistical modelling it is useful to introduce a further
parameter $\sigma$ that controls the speed of the process. This can
be done by defining $\Gfv'=\sigma\Gfv$, $\Gwf'=\sigma\Gwf$ and $\A_{\infty}'=\sigma\A_{\infty}$,
which correspond to the time change $X_{\tau(t)}$ with
$\tau(t)=\sigma t$.  In such parameterisation,
$\sigma$ does not affect the stationary distribution of the process,
and can be used to model the dependence structure. For simplicity of
exposition, the theory below focuses on the case $\sigma=1$.

%%%%%%%%%%%%%%%%%%%%%%%%%

\subsection{Filtering infinite Wright--Fisher diffusions}\label{subsec: FV-inf-alleles}

We consider a model for evolving distributions with countable
support. In particular,
suppose the signal $\XX_{t}$ follows an infinite WF diffusion with generator \eqref{eq:ethier-generator}, and assume that, given $\XX_{t}=\xx\in\Delta_{\infty}$, the emission distribution is multinomial with density
\begin{equation}\label{extended MN}
\text{MN}(\mm;|\mm|,\xx)={|\mm|\choose {\mm}}\prod_{j\ge1}x_j^{m_j}, \quad \quad |\mm|<\infty.
\end{equation}
Note that the above probability mass function is well defined for $\mm\in\Z^{\infty}$ such that $|\mm|<\infty$.
The following Theorem provides the prediction step of the
filtering algorithm, extending a result obtained in \cite{PR14} for finite $K$.

\begin{theorem}\label{prop: PR14-WF-infty}
Let $\XX_{t}$ have generator \eqref{eq:ethier-generator} and $\pi_{\aa}$ be the law of \eqref{infinite-dimensional dirichlet}. Then, for any $\mm\in\Z^{\infty}$ such that $|\mm|<\infty$,
\begin{equation}\label{nu-m propagation inf}
\psi_{t}\big(\pi_{\aa+\mm}\big)
=\sum_{\substack{\oo \le \ii\le \mm}}p_{\mm,\mm-\ii}(t)\pi_{\aa+\mm-\ii},
\end{equation}
where
\begin{equation}\label{transition probabilities}
\begin{split}
p_{\mm,\mm}(t)=&\,e^{-\lambda_{|\mm|}t},\\
p_{\mm,\mm-\ii }(t)=&\,
C_{|\mm| ,|\mm| -|\ii|}(t) \Bigg(\prod_{h=0}^{|\ii|-1}\lambda_{|\mm| -h}\Bigg)p(\ii;\,\mm,|\ii|),
\quad  \oo< \ii\le \mm,\\
C_{|\mm| ,|\mm| -|\ii|}(t)=&\,
(-1)^{|\ii|}\sum_{k=0}^{|\ii|}\frac{e^{-\lambda_{|\mm| -k}t}}{\prod_{0\le h\le |\ii|,h\ne k}(\lambda_{|\mm| -k}-\lambda_{|\mm| -h})},
\end{split}
\end{equation}
%are the transition probabilities of the death process with generator \eqref{WF dual gen}, with
%\begin{equation*}%\label{}
%C_{|\mm| ,|\mm| -|\ii|}(t)=
%(-1)^{|\ii|}\sum_{k=0}^{|\ii|}\frac{e^{-\lambda_{|\mm| -k}t}}{\prod_{0\le h\le |\ii|,h\ne k}(\lambda_{|\mm| -k}-\lambda_{|\mm| -h})},
%\end{equation*}
%is related to the convolution of the waiting times in an inhomogeneous Poisson process (cf.~Section 19.10 in \cite{JKB94}; see also \cite{SB99}),
and where
\begin{equation}\label{hypergeometric}
p(\ii;\,\mm,|\ii|)=\binom{|\mm| }{|\ii|}^{-1}
\prod_{j\ge1}\binom{m_{j}}{i_{j}}
\end{equation}
is the multivariate hypergeometric probability function, with parameters $(\mm,|\ii|)$, evaluated at $\ii$.
\end{theorem}

\begin{proof}
Define
\begin{equation}\label{h infinite alleles}
h(\xx,\mm)=\frac{\Gamma(\theta+|\mm|)}{\Gamma(\theta)}\prod_{j\ge1}\frac{\Gamma(\alpha_{j})}{\Gamma(\alpha_{j}+m_{j})}\xx^{\mm},
\quad \quad \mm\in\Z^{\infty},|\mm|<\infty,
\end{equation}
which is in the domain of $\A_{\infty}$, with $(\alpha_{1},\alpha_{2},\ldots)$ as in \eqref{inf-alpha constraints} and $m_{j}=0$ if $\alpha_{j}=0$. Let also $\ee_{i}$ be the vector whose only non zero component is a 1 at the $i$th coordinate. A direct computation shows that
\begin{equation}\label{eq:WF-gen-xm}
\begin{split}\notag
\A_{\infty} h(\xx,\mm)
=&\,
\sum_{i\ge1}\bigg(\frac{\alpha_{i}m_{i}}{2}+\binom{m_{i}}{2}\bigg)\frac{\Gamma(\theta+|\mm|)}{\Gamma(\theta)}\prod_{j\ge1}\frac{\Gamma(\alpha_{j})}{\Gamma(\alpha_{j}+m_{j})}\xx^{\mm-\ee_{i}}\\
&\,-\sum_{i\ge1}\bigg(\frac{\theta m_{i}}{2}+\binom{m_{i}}{2}+\frac{1}{2}m_{i}\sum_{j\ne i}m_{j}\bigg)\frac{\Gamma(\theta+|\mm|)}{\Gamma(\theta)}\prod_{j\ge1}\frac{\Gamma(\alpha_{j})}{\Gamma(\alpha_{j}+m_{j})}\xx^{\mm}\\
%=&\,\sum_{i=1}^{K}\lambda_{m_{i}}^{(i)}\xx^{\mm-\ee_{i}}-\lambda_{|\mm|}\xx^{\mm}.
=&\,\frac{\theta+|\mm|-1}{2}\sum_{i\ge1}m_{i}h(\xx,\mm-\ee_{i})-\frac{|\mm|(\theta+|\mm|-1)}{2}h(\xx,\mm).
\end{split}
\end{equation}
Hence, by \eqref{local duality}, the death process $M_{t}$ on $\Z^{\infty}$, which jumps from $\mm$ to $\mm-\ee_{i}$ at rate $m_{i}(\theta+|\mm|-1)/2$, is dual to the infinite Dirichlet diffusion with generator $\A_{\infty}$ with respect to \eqref{h infinite alleles}.
From the definition \eqref{prediction operator} of the prediction operator now we have
\begin{align*}%\label{}
\psi_{t}\big(\pi_{\aa+\mm}\big)(\d\xx')
=&\,\int_{\X}h(\xx,\mm)\pi_{\aa}(\d\xx)P_{t}(\xx,\d \xx')\\
=&\,\int_{\X}h(\xx,\mm)\pi_{\aa}(\d\xx')P_{t}(\xx',\d \xx)\\
=&\,\pi_{\aa}(\d \xx')\E^{\xx'}[h(\XX_{t},\mm)]\\
=&\,\pi_{\aa}(\d \xx')\E^{\mm}[h(\xx',M_{t})]\\
=&\,\pi_{\aa}(\d \xx')\sum_{\substack{\oo \le \ii\le \mm}}p_{\mm,\mm-\ii}(t)h(\xx',\mm-\ii)\\
=&\,\sum_{\substack{\oo \le \ii\le \mm}}p_{\mm,\mm-\ii}(t)\pi_{\aa+\mm-\ii}(\d\xx')
\end{align*}
where the second equality holds in virtue of the reversibility of
$\XX_{t}$ with respect to $\pi_{\aa}$, the fourth by the duality established above together with \eqref{duality identity} and the fifth from Lemma \ref{lemma: death transitions probabilities} in the Appendix.
\end{proof}

Assuming now an observation $\yy=(y_{1},y_{2},\ldots)$ has the extended multinomial likelihood \eqref{extended MN}, the update step \eqref{update operator} in this case becomes
\begin{equation}\label{inf-dir conjugacy}
\phi_{\yy}\big(\pi_{\aa+\mm}\big)=\pi_{\aa+\mm+\yy},
\end{equation}
which follows from the conjugacy of the Dirichlet process (see Section \ref{subsec: FV-static}) by taking $\Y=\N$.

The following Proposition summarises the findings of this section by providing the learning algorithm that allows computable filtering in this framework. Let
\begin{equation}\label{set of multiplicities}
\M = \{\mm=(m_1,\ldots,m_{K})\in\Z^{K},\, K\in\N\},
\end{equation}
with a partial ordering defined by ``$<$'', so that $\mm<\nn$ if
$m_{j}\le n_{j}$ for all $j\ge1$ and $m_{j}< n_{j}$ for some
$j\ge1$. For $M\subset \M$, let also
\begin{equation}\label{G(M)}
G(M)=\{\nn\in\M:\ \oo\le\nn\le\mm,\, \mm\in M\}
\end{equation}
be the set of nonnegative vectors lying beneath those in $M$.

\begin{proposition}\label{algorithm infinite alleles}
Consider the family of finite mixtures
\begin{equation*}%\label{}
\F=\bigg\{\sum_{\mm\in\Lambda}w_{\mm}h(\xx,\mm)\pi_{\aa}(\d \xx),\, \Lambda\subset\M,\ |\Lambda|<\infty,\ w_{\mm}\ge0,\ \sum_{\mm\in\Lambda}w_{\mm}=1\bigg\},
\end{equation*}
where $\pi_{\aa}$ is the law of \eqref{infinite-dimensional dirichlet} and $h(\xx,\mm)$ is as in \eqref{h infinite alleles}. Then, when $\XX_{t}$ has generator \eqref{eq:ethier-generator} and data are as in \eqref{extended MN}, $\F$ is closed under the application of the update and prediction operators \eqref{update operator} and \eqref{prediction operator}. Specifically,
\begin{equation}\label{Dir algorithm update}
  \phi_{\yy}\bigg ( \sum_{\mm \in \Lambda}  w_{\mm} h(\xx,\mm)
    \pi_{\aa}(\d \xx) \bigg)   = \sum_{\nn \in t(\yy,\Lambda)}  \hw_\nn  h(\xx,\nn)     \pi_{\aa}(\d \xx),
\end{equation}
with
\begin{equation}
  \label{eq:update-mix}
  \begin{split}
t(\yy,\Lambda)   :=&\, \{\nn: \nn=t(\yy,\mm), \mm \in \Lambda \} \\
\hw_\nn
  \propto&\, w_{\mm}
  \frac{\Gamma(\theta+|\mm|)}{\Gamma(\theta+|\mm|+|\yy|)}
  \prod_{j\ge1}  \frac{\Gamma(\alpha_{j}+\mm_{j}+y_{j})}{\Gamma(\alpha_{j}+\mm_{j})}
   \quad   \textrm{for}\quad  \nn = t(\yy,\mm) \,,\sum_{\nn
      \in t(\yy,\Lambda)} \hw_\nn =1\,,
  \end{split}
\end{equation}
and
\begin{equation}
  \label{eq:predict-mix}
  \psi_t\bigg ( \sum_{\mm \in \Lambda}  w_{\mm} h(\xx,\mm)
    \pi_{\aa}(\d \xx) \bigg) = \sum_{\nn \in G(\Lambda)} \bigg(
    \sum_{\substack{\mm \in \Lambda, \mm \geq \nn}} w_\mm
    p_{\mm,\nn}(t) \bigg) h(\xx,\nn) \pi_{\aa}(\d \xx).
\end{equation}
\end{proposition}
\begin{proof}
The update operation, given by \eqref{Dir algorithm update} and \eqref{eq:update-mix}, follows from \eqref{inf-dir conjugacy} together with the fact that for a mixture $\sum_{i=1}^n w_i \nu_i $ we have
\begin{equation}\label{update quasi linearity}
 \phi_{\yy}\bigg(\sum_{i=1}^n w_i \nu_i \bigg)
=\sum_{i=1}^n {w_i p_{\nu_i}(\yy) \over \sum_j w_j
p_{\nu_j}(\yy)} \phi_\yy(\nu_i),\quad \quad
p_{\nu_i}(\yy)=\int_{\X}f_{\xx}(\yy)\nu_{i}(\d \xx),
\end{equation}
while \eqref{eq:predict-mix} follows from the fact that
  \begin{equation}\label{prediction linearity}
 \psi_{t}\bigg(\sum_{i=1}^n w_i \nu_i \bigg)=\sum_{i=1}^n w_i  \psi_{t}(\nu_i)
\end{equation}
together with Theorem \ref{prop: PR14-WF-infty} and a rearrangement of the sums.
\end{proof}

%%%%%%%%%%%%%%%%%%%%%%%%%

\subsection{Filtering Fleming--Viot processes}\label{subsec: FV-FV}

Let now the signal $X_{t}$ be a FV process with generator \eqref{FV-generator}. We assume that given the signal state $X_{t}=x\in\MM_{1}$, observations are drawn independently from $x$, i.e.,
\begin{equation}\label{DP data}
y_{i}\mid x\overset{iid}{\sim}x.
\end{equation}
A sample $\yy_{1:m}=(y_{1},\ldots,y_{m})$ from the discrete distribution $x$ will feature ties
among the observations with positive probability. Denote by
$(y_{1}^{*},\ldots,y_{K_{m}}^{*})$ the distinct values in $\yy_{1:m}$ and by $\mm=(m_{1},\ldots,m_{K_{m}})$ the associated multiplicities, so that $|\mm|=m$.
%\begin{equation*}
%  \Y^{*} = \{\,\yy_{1:m}=(y_{1},\ldots,y_{K_{m}},0,0,\ldots),\  y_{j}\in \Y,\ y_{1}\ne \cdots\ne y_{K_{m}},\ K_{m}\le m\in\Z\,\}\,
%\end{equation*}
%to be the space of infinite vectors which record in order of appearance the $K_{m}\le m$ distinct values observed in an $m$-sized sample, and
The following result provides the prediction step for filtering FV
processes.

\begin{theorem}\label{prop: FV propagation}
Let $X_{t}$ have generator \eqref{FV-generator}, with invariant law $\Pi_{\alpha}$. Then, for any $\yy_{1:m}$ with multiplicities $\mm$,
\begin{equation}\label{FV propagation}
\psi_{t}\Bigg(\Pi_{\alpha+\sum_{i=1}^{K_{m}}m_{i}\delta_{y_{i}^{*}}}\Bigg)=
\sum_{\nn\in G(\mm)}p_{\mm,\nn}(t)
\Pi_{\alpha+\sum_{i=1}^{K_{m}}n_{i}\delta_{y_{i}^{*}}},
\end{equation}
with $p_{\mm,\nn}(t)$ as in \eqref{transition probabilities} and $G$ as in \eqref{G(M)}.
\end{theorem}
\begin{proof}
Fix an arbitrary partition $(A_{1},\ldots,A_{K})$ of $\Y$ with $K$ classes, and denote by $\tilde\mm$ the multiplicities resulting from binning $\yy_{1:m}$ into the corresponding cells. Then
\begin{equation}\label{DP to dir}
\Pi_{\alpha+\sum_{i=1}^{K_{m}}m_{i}\delta_{y_{i}^{*}}}(A_{1},\ldots,A_{K})
\sim
\pi_{\alpha+\tilde\mm},
\end{equation}
where $\Pi_{\alpha+\sum_{i=1}^{K_{m}}m_{i}\delta_{y_{i}^{*}}}(A_{1},\ldots,A_{K})$ denotes $\Pi_{\alpha+\sum_{i=1}^{K_{m}}m_{i}\delta_{y_{i}^{*}}}(\cdot)$ evaluated on $(A_{1},\ldots,A_{K})$.
Since the projection onto the same partition of the FV process is a $K$-dimensional WF process (see Section \ref{subsec: FV-model}), from Theorem \ref{prop: PR14-WF-infty} we have
\begin{align}\label{WF tilde propagation}
\psi_{t}\Big(\Pi_{\alpha+\sum_{i=1}^{K_{m}}m_{i}\delta_{y_{i}^{*}}}(A_{1},\ldots,A_{K})\Big)
=&\,\psi_{t}(\pi_{\alpha+\tilde\mm})\\
=&\,\sum_{\substack{\nn\in G(\tilde\mm)}}p_{\tilde\mm,\nn}(t)\pi_{\alpha+\nn}.\notag
\end{align}
Furthermore, since a Dirichlet process is characterised by its finite-dimensional projections, now it suffices to show that
\begin{equation*}%\label{}
\sum_{\nn\in G(\mm)}p_{\mm,\nn}(t)\Pi_{\alpha+\sum_{i=1}^{K_{m}}n_{i}\delta_{y_{i}^{*}}}(A_{1},\ldots,A_{K})
=\sum_{\substack{\nn\in G(\tilde\mm)}}p_{\tilde\mm,\nn}(t)\pi_{\alpha+\nn}
\end{equation*}
so that the operations of propagation and projection commute. Given \eqref{DP to dir}, we only need to show that the mixture weights are consistent with respect to fragmentation and merging of classes, that is
\begin{equation}\notag
\sum_{\ii\in G(\mm):\, \tilde \ii=\nn}p_{\mm,\ii}(t)=p_{\tilde\mm,\nn}(t),
\end{equation}
where $\tilde\ii$ denotes the projection of $\ii$ onto $(A_{1},\ldots,A_{K})$.
Using \eqref{transition probabilities}, the previous in turn reduces to
\begin{equation}\label{hypergeometric marginal}
\sum_{\ii\in G(\mm):\, \tilde \ii=\nn}p(\ii;\,\mm,m-i)=p(\nn;\,\tilde\mm,m-n),
\end{equation}
which holds by the marginalization properties of the multivariate hypergeometric distribution. Cf.~\cite{JKB97}, equation 39.3.
\end{proof}

Hence a single Dirichlet measure evolves into a finite mixture of Dirichlet measures. Note that when $\mm=(0,\ldots,0)$, \eqref{FV propagation} reduces to the stationarity equation for FV processes.

We now turn to the update step. Let $(y_{K_{m}+1}^{*},\ldots,y_{K_{m+n}}^{*})$ be the distinct values observed in an additional sample $\yy_{m+1:m+n}$ of size $n$ that are not included in $(y_{1}^{*},\ldots,y_{K_{m}}^{*})$, and let $\nn$ be the multiplicities of the full vector of distinct values $(y_{1}^{*},\ldots,y_{K_{m+n}}^{*})$.  %, and $\nn=(n_{K_{m}+1},\ldots,n_{K_{m+n}})$ the associated multiplicities.
Denote also by $\mathrm{PU}_{\alpha}(\yy_{m+1:m+n}\mid \yy_{1:m})$ the joint distribution of $\yy_{m+1:m+n}$ sampled from a conditional Blackwell--MacQueen P\'olya urn scheme \citep{BM73}, i.e.,
\begin{equation*}%\label{}
\begin{split}
&\,Y_{m+i+1} \mid \yy_{1:m+i} \sim
\frac{\theta P_{0}+\sum_{j=1}^{m+i}\delta_{y_{j}}}{\theta+m+i},\quad \quad i=0,\ldots,n-1.
\end{split}
\end{equation*}
The following result, stated here in our notation for ease of the reader, is a special case of \cite{A74}.

\begin{lemma}\label{prop: FV update multi}
Let $y_{m+i}\mid x\sim x$, $i=1,\ldots,n$, with
\begin{equation*}%\label{}
x\sim \sum_{\mm \in M}w_{\mm}\Pi_{\alpha+\sum_{i=1}^{K_{m}}m_{i}\delta_{y_{i}^{*}}},\quad \sum_{\mm\in M}w_\mm=1.
\end{equation*}
Then
\begin{align}\label{eq: NP-update}
\phi_{\yy_{m+1:m+n}}\Bigg(\sum_{\mm \in M}w_{\mm}\Pi_{\alpha+\sum_{i=1}^{K_{m}}m_{i}\delta_{y_{i}^{*}}}\Bigg)
=
\sum_{\nn\in t(\yy_{m+1:m+n},M)}\hat w_{\nn}\Pi_{\alpha+\sum_{i=1}^{K_{m+n}}n_{i}\delta_{y_{i}^{*}}},
\end{align}
where $t(\cdot)$ is as in Proposition \ref{algorithm infinite alleles} and
\begin{equation}\label{updated mixture weights}
\hat w_{\nn}\propto
w_{\mm}\,\mathrm{PU}_{\alpha}(\yy_{m+1:m+n}\mid \yy_{1:m}).
\end{equation}
\end{lemma}

\begin{proof}
The distribution $x$ is a mixture of Dirichlet processes with mixing measure $H(\cdot)=\sum_{\mm\in M}w_{\mm}\delta_{\mm}(\cdot)$ on $M$ and transition measure
\begin{equation*}%\label{}
\alpha_{\mm}(\cdot)
=\alpha(\cdot)+\sum_{j=1}^{K_{m}}m_{j}\delta_{y_{j}^{*}}(\cdot)
=\alpha(\cdot)+\sum_{i=1}^{m}\delta_{y_{i}}(\cdot),
\end{equation*}
where $\yy_{1:m}$ is the full sample.
See Section \ref{subsec: FV-static}. Lemma 1 and Corollary 3.2' in \cite{A74} now imply that
\begin{equation*}%\label{}
x\mid \mm,\yy_{m+1:m+n}\sim
\Pi_{\alpha_{\mm}(\cdot)+\sum_{i=m+1}^{n}\delta_{y_{i}}(\cdot)}
=\Pi_{\alpha(\cdot)+\sum_{i=1}^{n}\delta_{y_{i}}}
\end{equation*}
and $H(\mm\mid \yy_{m+1:m+n})\propto w_{\mm}\,\mathrm{PU}_{\alpha}(\yy_{m+1:m+n}\mid \yy_{1:m})$.
\end{proof}

Hence, the updated mixture of Dirichlet processes is still a mixture of
Dirichlet processes with different multiplicities and possibly new
point masses in the parameter measures. Iterating the propagation and
update operations provided by Theorem \ref{prop: FV propagation} and
Lemma \ref{prop: FV update multi} yields the computable filter
for a partially observed FV process, which sequentially evaluates
$\L(X_{t_n} | Y_{1:n})$. This is summarised in the next Proposition.

\begin{proposition}\label{algorithm FV}
Consider the family of finite mixtures of Dirichlet processes
\begin{align}\notag%\label{class DP finite mixtures}
\Fdir=
\Bigg\{\sum_{\mm\in M}w_{\mm}\Pi_{\alpha+\sum_{i=1}^{K_{m}}m_{i}\delta_{y_{i}^{*}}} :\ M \subset \M,\ |M|<\infty,\, w_{\mm}\ge0,\  \sum_{\mm\in M}w_\mm=1\Bigg\}.
\end{align}
Then, when $X_{t}$ has generator \eqref{FV-generator} and data are as in \eqref{DP data}, $\Fdir$ is closed under the application of the update and prediction operators \eqref{update operator} and \eqref{prediction operator}. Specifically,
\begin{align}\label{eq: FV algorithm update}
\phi_{\yy_{m+1:m+n}}\Bigg(\sum_{\mm \in M}w_{\mm}\Pi_{\alpha+\sum_{i=1}^{K_{m}}m_{i}\delta_{y_{i}^{*}}}\Bigg)
=
\sum_{\nn\in t(\yy_{m+1:m+n},M)}\hat w_{\nn}\Pi_{\alpha+\sum_{i=1}^{K_{m+n}}n_{i}\delta_{y_{i}^{*}}},
\end{align}
with
\begin{equation}
  \label{eq:update-mix FV}\notag
  \begin{split}
t(\yy,\Lambda) :=&\, \{\nn: \nn=t(\yy,\mm), \mm \in \Lambda \} \\
\hat w_{\nn}\propto&\,
w_{\mm}\,\mathrm{PU}_{\alpha}(\yy_{m+1:m+n}\mid \yy_{1:m})
   \quad   \textrm{for}\quad  \nn = t(\yy,\mm) \,,\sum_{\nn
      \in t(\yy,\Lambda)} \hw_\nn =1\,,
  \end{split}
\end{equation}
and
\begin{equation}\label{eq: FV algorithm propagation}
\psi_{t}\Bigg(\sum_{\mm\in M}w_{\mm}\Pi_{\alpha+\sum_{i=1}^{K_{m}}m_{i}\delta_{y_{i}^{*}}}\Bigg)=
\sum_{\nn\in G(M)}
\Bigg(\sum_{\mm\in M,\, \mm\ge \nn}w_{\mm}p_{\mm,\nn}(t)\Bigg)
\Pi_{\alpha+\sum_{i=1}^{K_{m}}n_{i}\delta_{y_{i}^{*}}}.
\end{equation}
\end{proposition}

\begin{proof}
The update operation \eqref{eq: FV algorithm update} follows directly from Lemma \ref{prop: FV update multi}.
The prediction operation \eqref{eq: FV algorithm propagation} for elements of $\F_{\Pi}$ follows from Theorem \ref{prop: FV propagation} together with \eqref{prediction linearity} and a rearrangement of the sums, so that
\begin{align*}%\label{prediction sum change}
\psi_{t}\Bigg(\sum_{\mm\in M}w_{\mm}\Pi_{\alpha+\sum_{i=1}^{K_{m}}m_{i}\delta_{y_{i}^{*}}}\Bigg)
=&\,\sum_{\mm\in M}w_{\mm}
\sum_{\substack{\nn\in G(\mm)} }
p_{\mm,\nn }(t)
\Pi_{\alpha+\sum_{i=1}^{K_{m}}n_{i}\delta_{y_{i}^{*}}}\\
=&\,\sum_{\nn\in G(M)}
\Bigg(\sum_{\mm\in M,\,\mm\ge \nn}w_{\mm}p_{\mm,\nn}(t)\Bigg)
\Pi_{\alpha+\sum_{i=1}^{K_{m}}n_{i}\delta_{y_{i}^{*}}}.
\end{align*}
\end{proof}

%%%%%%%%%%%%%%%%%%%%%%%%%

%\subsection{Complexity}\label{subsec: FV-complexity}

%%%%%%%%%%%%%%%%%%%%%%%%%%%%%

\section{Filtering Dawson--Watanabe processes}\label{sec: DW}

\subsection{The static model: gamma random measures}\label{subsec: DW-static}

Gamma random measures can be thought of as the counterpart of Dirichlet processes in the context of finite intensity measures. A gamma random measure $z\in\MM(\Y)$ with shape parameter $\alpha$ as in \eqref{alpha measure} and rate parameter $\beta>0$, denoted $z\sim \Gamma^{\beta}_{\alpha}$, admits representation
\begin{equation}\label{gamma process series}
z(\cdot)=\beta^{-1}\sum_{i=1}^{\infty}Q_{i}\delta_{Y_{i}}(\cdot), \quad  Y_{i}\overset{iid}{\sim}P_{0},
\end{equation}
with $\{Q_{i}, i\ge1\}$ as in \eqref{DP series representation}. The conjugacy and projection properties for gamma random measures are as follows:
\begin{itemize}
\item \emph{Conjugacy}:
gamma random measures are conjugate with respect to Poisson point processes data. Let $N$ be a Poisson point process on $\Y$ with random intensity measure $z\sim\Gamma^{\beta}_{\alpha}$, i.e., conditionally on $z$, $N(A_i)\overset{ind}{\sim}\text{Po}(z(A_{i}))$ for any disjoint sets $A_1,\ldots,A_K\in\Y$, $K\in\N$. Let $m:=N(\Y)$ and $y_i:=N(A_i)/N(\Y)$, $i=1,\ldots,K$, so that
\begin{equation}\label{poisson data}
y_{i}\mid z,m\overset{iid}{\sim}z/|z|,\quad \quad
m\mid z\sim\text{Po}(|z|),\quad \quad
z\sim \Gamma^{\beta}_{\alpha},
\end{equation}
where $|z|:=z(\Y)$ is the total mass of $z$ (recall that $z\in\MM(\Y)$ is a finite measure). Then
\begin{equation}\label{gamma conjugacy}
z\mid y_{1:m}\sim \Gamma^{\beta+m}_{\alpha+\sum_{i=1}^{m}\delta_{y_{i}}}.
\end{equation}
\item \emph{Projection}: for a measurable partition $A_{1},\ldots,A_{K}$ of  $\Y$, the vector $(z(A_{1}),\ldots,z(A_{K}))$ has independent components $z(A_{i})$ with gamma distribution $\text{Ga}(\alpha_{i},\beta)$, $\alpha_{i}=\alpha(A_{i})$.
\end{itemize}
The above conjugacy property was showed by \cite{L82}.
%A dynamic version of \eqref{gamma-dirichlet algebra} is also satisfied by the DW and the FV process, up to an appropriate random time change. See Section \ref{sec: time changes} for references.
%\quad \quad\text{if }z\sim\Gamma^{\beta}_{\alpha}\text{ and }y\mid z\sim z;
%Alternatively, this can be seen from disintegration of the measure $z\sim\Gamma^{\beta}_{\alpha}$, where $z=z(\Y)x$, $z(\Y)\sim\Gamma(\alpha,\beta)$ and $x\sim\Pi_{\alpha}$, assuming a Poisson number of \emph{iid} observation sampled from $x$, or more precisely \ch\red{\eqref{poisson data}}
%Then \eqref{gamma conjugacy} follows from \eqref{DP conjugacy} together with the usual Bayesian conjugacy for gamma distributions with Poisson likelihood.
Finally, it is well known that \eqref{DP series representation} and \eqref{gamma process series} satisfy the following relation in distribution
\begin{equation}\label{gamma-dirichlet algebra}
x(\cdot)=\frac{z(\cdot)}{z(\Y)}\sim\Pi_\alpha,
\end{equation}
where $x$ is independent of $z(\Y)$, which extends to the
infinite dimensional case the well known relationship between beta and
gamma random variables. See for example \cite{DVJ08}, Example 9.1(e).
See also \cite{F10}, Section 1.3, for a dynamic analog of \eqref{gamma-dirichlet algebra} linking CIR and WF processes and \cite{P91} for a measure-valued version.

%%%%%%%%%%%%%%%%%%%%%

\subsection{The signal: Dawson--Watanabe processes}\label{subsec: DW-model}

 Dawson--Watanabe (DW) processes are branching measure-valued diffusions taking values in the space of finite discrete measures. See \cite{D93} and \cite{L11} for reviews.
Here we are interested in the special case of subcritical branching
with immigration, where subcriticality refers to the fact that the mean number of offspring per individual in the underlying population is less than one. We will consider DW processes with generator
\begin{align}\label{DW-generator}
\Gdw \varphi(z)
=&\,\frac{1}{2}\int_{\Y}z(\d y)\frac{\partial^{2}\varphi(z)}{\partial z(\d y)^{2}}
+\int_{\Y}z(\d y)\tilde B\bigg(\frac{\partial\varphi(z)}{\partial z(\cdot)}\bigg)(y),
%\Gdw \varphi_{m}(x)
%=&\,\sum_{i=1}^{m}\la I_{i}f,x^{m}\ra
%+\frac{1}{2}\sum_{\substack{1\le k\ne i\le m}}\la\red{ \Psi_{ki}f-f},x^{m}\ra,
\end{align}
where $\tilde B$ is
\begin{equation}\label{immigration operator}\notag
(\tilde Bf)(y)=\frac{\na}{2}\int_{\Y}f(u)P_{0}(\d y')-\frac{\beta}{2}f(y)
,\quad \quad f\in C(\Y),
\end{equation}
 and the domain of $\Gdw$ is as in \eqref{FV-generator} except that $C^{2}(\R^{k})$ is replaced by its subspace of functions with compact support. Contrary to \eqref{FV-generator}, whose first term describes substitution of individuals in the underlying population, the first term in $\Gdw$ describes addition of individuals through branching, whereas the second accounts for independent immigration of individuals, whose type is selected according to the probability distribution $P_0$. These heuristics provide intuition for the fact that DW processes drive evolving measures with non constant mass.
%\begin{equation}\label{immigration operator}\notag
%Ig(y)=\frac{\na}{2}\int_{\Y}g(y')P_{0}(\d y'),\quad \quad g\in B(\Y),
%\end{equation}
%applied to the $i$th coordinate of $f$, $P_{0}$ is the nonatomic immigrants distribution, and
%\begin{equation*}%\label{}
%\Psi_{ij}f(y_{1},\ldots,y_{m})=(y_{1},\ldots,y_{i-1},y_{k},y_{i},\ldots,y_{m})
%\end{equation*}
%defines the \emph{resampling} mechanism.
The DW process with the above operator is known to be stationary and reversible with respect to the law $\Gamma^{\beta}_{\alpha}$ of a gamma random measure as in \eqref{gamma process series} \citep{S90,EG93b}.

Let $Z_{t}$ have generator \eqref{DW-generator}. Given a measurable partition $A_{1},\ldots,A_{K}$ of $\Y$, the vector $(Z_{t}(A_{1}),\ldots,Z_{t}(A_{K}))$ has independent components $Z_{t}(A_{i})$ each driven by a Cox--Ingersoll--Ross (CIR) diffusion \citep{CIR85}. These are also subcritical continuous-state branching processes with immigration, reversible and ergodic with respect to a $\text{Ga}(\alpha_{i},\beta)$ distribution, with generator
\begin{equation}\label{CIR generator}
\Gcir_{i}f(z_{i})=\frac{1}{2}(\alpha_{i}-\beta z_{i})f'(z_{i})+\frac{1}{2}z_{i}f''(z_{i}),
\end{equation}
acting on $C^{2}([0,\infty))$ functions which vanish at infinity. See \cite{KW71}.
%Laplace-type duality results, that is with respect to functions of the type $h(x,y) = e^{-axy}$, are known for CIR processes as in \eqref{CIR generator}. See, for example, \cite{E00} \ch\red{(to be checked!)}.
%\ch\red{As recalled in more detail below, these give rise, under projections onto finite-dimensional partitions, to independent ergodic continuous state branching processes with immigration \citep{J58,L67}. }

%\begin{remark}
As for FV and WF processes, a parameter that controls the speed could also be introduced in this case. This corresponds to the original parametrisation by \cite{CIR85}, whereby the process has generator
\begin{equation*}%\label{}
\mathcal{C}f(z)=\kappa(\vartheta- z)f'(z)+\frac{\sigma^{2}}{2}zf''(z),
\end{equation*}
and invariant distribution $\text{Ga}(2\kappa\vartheta/\sigma^{2},2\kappa/\sigma^{2})$. Here, for simplicity of exposition, we have set $\alpha=2\kappa\vartheta/\sigma^{2}$, $\beta=2\kappa/\sigma^{2}$ and $\sigma^{2}=1$.
%\end{remark}

%%%%%%%%%%%%%%%%%%%%%%%%%

\subsection{Duality and propagation for multivariate CIR signals}\label{subsec: DW-inf-alleles}

Assume the signal $\ZZ_{t}=(Z_{1,t},\ldots,Z_{K,t})$ has independent CIR components $Z_{i,t}$ with generator \eqref{CIR generator}. The next proposition identifies the dual process for $\ZZ_{t}$.

\begin{theorem}\label{multi CIR duality}
Let $Z_{i,t}$, $i=1,\ldots,K$, be independent CIR processes each with generator \eqref{CIR generator} parametrised by $(\alpha_i,\beta)$, respectively. For $\aa\in\R_{+}^{K}$ and $\theta=|\aa|$, define $h_{\alpha_{i}}^{C}:\R_{+}\times\Z\times\R_{+}$ as
\begin{equation}\label{CIR duality function}\notag
h_{\alpha_{i}}^{C}(z,m,\ss)=
\frac{\Gamma(\alpha_{i})}{\Gamma(\alpha_{i}+m)}
\bigg(\frac{\beta+\ss}{\beta}\bigg)^{\alpha_{i}}
(\beta+\ss)^{m}z^{m}e^{-\ss z},
\end{equation}
$h^{W}:\R_{+}^{K}\times\Z^{K}$ as
\begin{equation}\label{WF duality function}\notag
h^{W}(\xx,\mm)=\frac{\Gamma(\na+|\mm|)}{\Gamma(\na)}\prod_{i=1}^{K}\frac{\Gamma(\alpha_{i})}{\Gamma(\alpha_{i}+m_{i})}\xx^{\mm},
\end{equation}
and $h:\R_{+}^{K}\times\Z^{K}\times\R_{+}$ as
\begin{equation}\label{multi cir duality function}\notag
h(\zz,\mm,\ss)=h_\theta^{C}(|\zz|,|\mm|,\ss)h^{W}(\xx,\mm),
\end{equation}
where $\xx=\zz/|\zz|$.
Then the joint process $\{(Z_{1,t},\ldots,Z_{K,t}),t\ge0\}$ is dual, in the sense of \eqref{duality identity}, to the process $\{(\mathbf{M}_{t},\SS_{t}),t\ge0\}\subset\Z^{K}\times\R_{+}$ with generator
\begin{equation}\label{K-CIR dual generator}
\Gn g(\mm,\ss)
  = \frac{1}{2}|\mm|(\beta+\ss) \sum_{i=1}^{K} \frac{m_i}{|\mm|}
[   g(\mm-\ee_{i},\ss)-g(\mm,\ss)]
-\frac{1}{2}\ss(\beta+\ss)   \frac{\partial g(\mm,\ss)}{\partial \ss}
\end{equation}
with respect to $h(\zz,\mm,\ss)$.
\end{theorem}

\begin{proof}
Throughout the proof, for ease of notation we will write $h_i^C$ instead of $h_{\alpha_i}^C$. Note first that for all $\mm\in \Z^{K}$ we have
\begin{equation}\label{composition duality function}
\prod_{i=1}^{K}h_i^{C}(z_{i},m_{i},\ss)
=h_{\theta}^{C}(\norm{z},|\mm|,\ss)h^{W}(\xx,\mm),
\end{equation}
where $x_{i}=z_{i}/\norm{z}$,
which follows from direct computation by multiplying and dividing by the correct ratios of gamma functions and by writing $\prod_{i=1}^{K}z_{i}^{m_{i}}=|z|^{m}
\prod_{i=1}^{K}x_{i}^{m_{i}}$.
We show the result for $K=2$, from which the statement for general $K$ case follows easily.
From the independence of the CIR processes, the generator $(Z_{1,t},Z_{2,t})$ applied to the left hand side of \eqref{composition duality function} is
\begin{align}\label{2CIR generator}
(\Gcir_{1}+\Gcir_{2})h^{C}_{1}h^{C}_{2}
=&\,h^{C}_{2}\Gcir_{1}h^{C}_{1}+h^{C}_{1}\Gcir_{2}h^{C}_{2}.
\end{align}
A direct computation shows that
\begin{align*}%\label{}
\Gcir_{i}h^{C}_{i}
=&\,
\frac{m_{i}}{2}(\beta+\ss)h^{C}_{i}(z_{i},m_{i}-1,\ss)
+\frac{\ss}{2}(\alpha_{i}+m_{i})h^{C}_{i}(z_{i},m_{i}+1,\ss)\notag\\
&\,-\frac{1}{2}(\ss(\alpha_{i}+m_{i})+m_{i}(\beta+\ss))h^{C}_{i}(z_{i},m_{i},\ss).
\end{align*}
Substituting in the right hand side of \eqref{2CIR generator} and collecting terms with the same coefficients gives
\begin{align*}%\label{}
&\frac{\beta+\ss}{2}\Big[m_{1}h^{C}_{1}(z_{1},m_{1}-1,\ss)h^{C}_{2}(z_{2},m_{2},\ss)
 +m_{2}h^{C}_{1}(z_{1},m_{1},\ss)h^{C}_{2}(z_{2},m_{2}-1,\ss)\Big]\\
&\,
+ \frac{\ss}{2}\Big[(\alpha_{1}+m_{1})h^{C}_{1}(z_{1},m_{1}+1,\ss)h^{C}_{2}(z_{2},m_{2},\ss)
 +(\alpha_{2}+m_{2})h^{C}_{1}(z_{1},m_{1},\ss)h^{C}_{2}(z_{2},m_{2}+1,\ss)\Big]\\
&\,-\frac{1}{2}(\ss(\alpha+m)+m(\beta+\ss))h^{C}_{1}(z_{1},m_{1},\ss)h^{C}_{2}(z_{2},m_{2},\ss)\notag
\end{align*}
with $\alpha=\alpha_{1}+\alpha_{2}$ and $m=m_{1}+m_{2}$.
From \eqref{composition duality function} we now have
\begin{align}\notag
%(\Gcir_{1}+\Gcir_{2})h^{C}_{1}h^{C}_{2}=
&\,
 \frac{\beta+\ss}{2}h_\theta^{C}(|z|,m-1,\ss)\Big[
 m_{1}h^{W}(\xx,\mm-\ee_{1},\ss)
 +m_{2}h^{W}(\xx,\mm-\ee_{2},\ss)\Big]\\
&\,
+ \frac{\ss}{2} h_\theta^{C}(|z|,m+1,\ss)\Big[(\alpha_{1}+m_{1})h^{W}(\xx,\mm+\ee_{1},\ss)+(\alpha_{2}+m_{2})h^{W}(\xx,\mm+\ee_{2},\ss)\Big]\notag\\
&\,-\frac{1}{2}(\ss(\alpha+m)+m(\beta+\ss))
h_\theta^{C}(|z|,m,\ss)h^{W}(\xx,\mm,\ss)\notag.
\end{align}
Then
\begin{align}\label{gen for duality check}
\begin{split}
(\Gcir_{1}+\Gcir_{2})h^{C}_{1}h^{C}_{2}=&\,
\frac{\beta+\ss}{2}\Big[m_{1}h(\zz,\mm-\ee_{1},\ss)+m_{2}h(\zz,\mm-\ee_{2},\ss)\Big]\\
&\,
+ \frac{\ss}{2}\Big[(\alpha_{1}+m_{1})h(\zz,\mm+\ee_{1},\ss)+(\alpha_{2}+m_{2})h(\zz,\mm+\ee_{2},\ss)\Big]\\
&\,-\frac{1}{2}(\ss(\alpha+m)+m(\beta+\ss))
h(\zz,\mm,\ss).
\end{split}\end{align}
Noting now that
\begin{align*}%\label{}
\frac{\partial}{\partial \ss}h(\zz,\mm,\ss)
%\frac{\partial}{\partial \ss}&\,
%\bigg(e^{-\ss|\xx|}
%(\beta+\ss)^{|\aa|+|\mm| }
%\beta^{-|\aa|}
%\prod_{i=1}^{2}\frac{\Gamma(\alpha_{i})}{\Gamma(\alpha_{i}+m_{i})}
%x_{i}^{m_{i}}\bigg)\\
%=&\,\frac{|\aa|+|\mm| }{\beta+\ss}h(\xx,\mm,\ss)
%-\frac{x_{1}+x_{2}}{\beta+\ss}
%(\beta+\ss)^{|\aa|+|\mm| +1}
%e^{-\ss|\xx|}\beta^{-|\aa|}
%\prod_{i=1}^{2}\frac{\Gamma(\alpha_{i})}{\Gamma(\alpha_{i}+m_{i})}
%x_{i}^{m_{i}}
%\\
=&\,\frac{\alpha+m}{\beta+\ss}
h(\zz,\mm,\ss)
-\frac{\alpha_{1}+m_{1}}{\beta+\ss}h(\zz,\mm+\ee_{1},\ss)
-\frac{\alpha_{2}+m_{2}}{\beta+\ss}h(\zz,\mm+\ee_{2},\ss),
\end{align*}
an application of \eqref{K-CIR dual generator} on $h(\zz,\mm,\ss)$
shows that $(\Gn h(\zz,\cdot,\cdot))(\mm,\ss)$ equals the right hand side of \eqref{gen for duality check},
%\begin{equation*}%\label{}
%((\Gcir_{1}+\Gcir_{2})h^{C}_{1}(\cdot,m_{1},\ss)h^{C}_{2}(\cdot,m_{2},\ss))(z_{1},z_{2})
%=(\Gn h(\zz,\cdot,\cdot))(\mm,\ss)
%\end{equation*}
so that \eqref{local duality} holds, giving the result.
\end{proof}

The previous Theorem extends the gamma-type duality showed for one dimensional CIR processes in \cite{PR14}.
Here $\textbf{M}_{t}$ is a $K$-dimensional death process $\textbf{M}_{t}\subset \Z^{K}$ which, conditionally on $\SS_{t}$, jumps from $\mm$ to $\mm-\ee_{i}$ at rate $2m_{i}(\beta+\SS_{t})$, and $\SS_{t}\in \R_{+}$ is a nonnegative deterministic process driven by the logistic type differential equation
\begin{equation}\label{ODE}
\frac{\d \SS_{t}}{\d t}=-\frac{1}{2}\SS_{t}(\beta+\SS_{t}).
\end{equation}
The next Proposition formalises the propagation step for multivariate CIR processes. Denote by $\textbf{Ga}(\aa,\beta)$ the product of gamma distributions $\text{Ga}(\alpha_{1},\beta)\times\cdots\times\text{Ga}(\alpha_{K},\beta)$, with $\aa=(\alpha_{1},\ldots,\alpha_{K})$.

%\begin{equation}\label{gamma vector}\notag
%\textbf{Ga}(\aa,\beta)=
%[\text{Ga}(\alpha_{1},\beta),\ldots,\text{Ga}(\alpha_{K},\beta)].
%%\quad \quad
%%z_{j}\sim \text{Ga}(\alpha_{j},\beta),
%%\quad \quad
%%y\mid z_{j}\sim \text{Po}(z_{j}).
%\end{equation}

\begin{proposition}\label{K-CIR implications}
Let $\{(Z_{1,t},\ldots,Z_{K,t}),t\ge0\}$ be as in Theorem \ref{multi CIR duality}. Then
\begin{align}\label{multi CIR propagation}
\psi_{t}\big(&\mathbf{Ga}(\aa+\mm,\beta+\ss)\big)=\\
=&\,\sum_{i=0}^{|\mm|}\mathrm{Bin}(|\mm|-i;\,|\mm|,p(t))\mathrm{Ga}(\na+|\mm|-i,\beta+\SS_{t})\sum_{\substack{\oo \le \ii\le \mm, |\ii|=i}}p(\ii;\, \mm,i)
\pi_{\aa+\mm-\ii},\notag
\end{align}
where
\begin{equation}\label{bernoulli parameter}
p(t)=\frac{\beta}{(\beta+\ss)e^{\beta t/2}-\ss},\quad
\SS_{t}=\frac{\beta s}{(\beta+s)e^{\beta t/2}-s}, \quad \SS_{0}=\ss
\end{equation}
and $p(\ii;\, \mm,i)$ is as in \eqref{hypergeometric}.
\end{proposition}

\begin{proof}
From independence we have
\begin{align*}%\label{}
\psi_{t}\big(\mathbf{Ga}(\aa+\mm,\beta+\ss)\big)
=&\,
\prod_{i=1}^{K}\psi_{t}\big(\mathrm{Ga}(\alpha_{i}+m_{i},\beta+\ss)\big).
\end{align*}
Using Lemma \ref{lemma: 1-CIR propagation} in the Appendix, the previous equals
\begin{align*}%\label{}
\prod_{i=1}^{K}&\sum_{j=0}^{m_{i}}\mathrm{Bin}(m_{i}-j;\,m_{i},p(t))
\mathrm{Ga}(\alpha_{i}+m_{i}-j,\beta+\SS_{t})\\
=&\,
\sum_{i_{1}=0}^{m_{1}}\mathrm{Bin}(m_{1}-i_{1};\,m_{1},p(t))
\mathrm{Ga}(\alpha_{1}+m_{1}-i_{1},\beta+\SS_{t})\\
&\,\times\cdots\times
\sum_{i_{K}=0}^{m_{K}}\mathrm{Bin}(m_{K}-i_{K};\,m_{K},p(t))
\mathrm{Ga}(\alpha_{K}+m_{K}-i_{K},\beta+\SS_{t}).
\end{align*}
%We show it for $K=2$, from which the general case follows easily. From Proposition \ref{prop: PR14-CIR} we have
%\begin{align*}
%\psi_{t}\big(&\text{Ga}(\alpha_{1}+m_{1},\beta+\ss)\text{Ga}(\alpha_{2}+m_{2},\beta+\ss)\big)=\\
%=&\,
%\sum_{i_{1}=0}^{m_{1}}\mathrm{Bin}(m_{1}-i_{1};\,m_{1},p(t))\mathrm{Ga}(\alpha_{1}+m_{1}-i_{1},\beta+\SS_{t})\\
%&\,\times\sum_{i_{2}=0}^{m_{2}}\mathrm{Bin}(m_{2}-i_{2};\,m_{2},p(t))\mathrm{Ga}(\alpha_{2}+m_{2}-i_{2},\beta+\SS_{t})\\
%=&\,\sum_{i=0}^{m_{1}+m_{2}}\mathrm{Bin}(m_{1}+m_{2}-i;\,m_{1}+m_{2},p(t))\mathrm{Ga}(\na+m_{1}+m_{2}-i,\beta+\SS_{t})\\
%&\,\sum_{\substack{(0,0)\le (i_{1},i_{2})\le (m_{1},m_{2})
%:\ i_{1}+i_{2}=i}}p(i_{1},i_{2};\, m_{1},m_{2},i)\text{Beta}(\alpha_{1}+m_{1}-i_{1},\alpha_{2}+m_{2}-i_{2}),\notag
%\end{align*}
%where $\theta=|\aa|$, $\text{Beta}(a,b)$ denotes a Beta distribution with parameters $(a,b)$ and the last identity follows from the fact that $\text{Ga}(\alpha_{1},\beta)\text{Ga}(\alpha_{2},\beta)=\text{Ga}(\alpha_{1}+\alpha_{2},\beta)\text{Beta}(\alpha_{1},\alpha_{2})$ and the fact that
Using now the fact that a product of Binomials equals the product of a Binomial and an hypergeometric distribution, we have
\begin{equation*}%\label{}
\sum_{i=0}^{|\mm|}\mathrm{Bin}(|\mm|-i;\,|\mm|,p(t))\sum_{\substack{\oo \le \ii\le \mm, |\ii|=i}}p(\ii;\, \mm,i)
\prod_{j=1}^{K}\mathrm{Ga}(\alpha_{j}+m_{j}-i_{j},\beta+\SS_{t})
\end{equation*}
which, using \eqref{gamma-dirichlet algebra}, yields \eqref{multi CIR propagation}. Furthermore, \eqref{bernoulli parameter} is obtained by solving \eqref{ODE} and by means of the following argument. The one dimensional death process that drives $|\textbf{M}_{t}|$ in Theorem \ref{multi CIR duality}, jumps from $|\mm|$ to $|\mm|-1$ at rate $|\mm|(\beta+S_{t})/2$, see \eqref{K-CIR dual generator}. Hence the probability of this event not occurring in $[0,t]$ is
\begin{equation*}%\label{}
P(|\mm|\mid |\mm|,S_{t})=\exp\bigg\{-\frac{|\mm|}{2}\int_{0}^{t}(\beta+S_{u})\d u\bigg\}=\left(\frac{\beta}{(\beta+\ss)e^{\beta t/2}-\ss}\right)^{|\mm|}.
\end{equation*}
The probability of a jump from $|\mm|$ to $|\mm|-1$ occurring in $[0,t]$ is
\begin{align*}%\label{}
P(|\mm|-1&\mid |\mm|,S_{t})\\
=&\,\int_{0}^{t}
\exp\bigg\{-\frac{|\mm|}{2}\int_{0}^{s}(\beta+S_{u})\d u\bigg\}
\frac{|\mm|}{2}S_{s}
\exp\bigg\{-\frac{|\mm|-1}{2}\int_{s}^{t}(\beta+S_{u})\d u\bigg\}\d t\\
=&\,\frac{|\mm|}{2}
\exp\bigg\{-\frac{|\mm|}{2}\int_{0}^{t}(\beta+S_{u})\d u\bigg\}
\int_{0}^{t}S_{s}
\exp\bigg\{\bigg(\frac{|\mm|}{2}-\frac{|\mm|-1}{2}\bigg)\int_{s}^{t}(\beta+S_{u})\d u\bigg\}\d t\\
=&\,|\mm|
\exp\bigg\{-\frac{|\mm|}{2}\int_{0}^{t}(\beta+S_{u})\d u\bigg\}
\bigg(1-\exp\bigg\{\bigg(\frac{|\mm|}{2}-\frac{|\mm|-1}{2}\bigg)\int_{0}^{t}(\beta+S_{u})\d u\bigg\}\bigg)\\
=&\,|\mm|
\bigg(\exp\bigg\{-\frac{|\mm|}{2}\int_{0}^{t}(\beta+S_{u})\d u\bigg\}-\exp\bigg\{-\frac{|\mm|-1}{2}\int_{0}^{t}(\beta+S_{u})\d u\bigg\}\bigg)\\
=&\,|\mm|\left(\frac{\beta}{(\beta+\ss)e^{\beta t/2}-\ss}\right)^{|\mm|-1}\left(1-\frac{\beta}{(\beta+\ss)e^{\beta t/2}-\ss}\right).
\end{align*}
Iterating the argument leads to conclude that the death process jumps from $|\mm|$ to $|\mm|-i$ in $[0,t]$ with probability $\text{Bin}(|\mm|-i\mid |\mm|,p(t))$.
\end{proof}

Note that when $s\in\N$, $\text{Ga}(\alpha_{i}+m,\beta+\ss)$ is the posterior distribution of a parameter with $\text{Ga}(\alpha_{i},\beta)$ distribution, given $s$ Poisson observations whose sum is $m$. Hence the dual component $M_{i,t}$ is interpreted as the sum of the observed values of type $i$, and $\SS_{t}\subset\R_{+}$ as a continuous version of the sample size.

%%%%%%%%%%%%%%%%%%%%%%%%%

\subsection{Filtering Dawson--Watanabe processes}\label{subsec: DW-DW}

Let now the signal $Z_{t}$ follow a DW process with generator \eqref{DW-generator}, and let $\Gamma^{\beta}_{\alpha}$ be the law of a gamma random measure, defined in \eqref{gamma process series}.
%Given a sample $\yy_{1:m}=(y_{1},\ldots,y_{m})$ as in \eqref{poisson data}, let $\mm\cdot\dd_{\yy_{1:m}^{*}}=\sum_{j=1}^{K_{m}}m_{j}\delta_{y_{j}^{*}}$ where as above $\yy_{1:m}^{*}$ are the distinct values $(y_{1}^{*},\ldots,y_{K_{m}}^{*})$ in $\yy_{1:m}$.
%Finally,
The following result provides the propagation step for filtering the DW process.

\begin{theorem}\label{prop: DW propagation}
Let $p(t)$ and $\SS_{t}$ be as in \eqref{bernoulli parameter} and $p(\nn;\,\mm,|\nn|)$ as in \eqref{hypergeometric}. Then
\begin{equation}\label{gamma propagation}
\psi_{t}\Big(\Gamma^{\beta+\ss}_{\alpha+\sum_{i=1}^{K_{m}}m_{i}\delta_{y_{i}^{*}}}\Big)=
\sum_{\nn\in G(\mm)}
\tilde p_{\mm,\nn}(t)
\Gamma^{\beta+\SS_{t}}_{\alpha+\sum_{i=1}^{K_{m}}n_{i}\delta_{y_{i}^{*}}},
\end{equation}
where
\begin{equation}\label{tilde pmn}
\tilde p_{\mm,\nn}(t)=\mathrm{Bin}(|\mm|-|\nn|;\,|\mm|,p(t))p(\nn;\, \mm,|\nn|),
\end{equation}
and $G(M)$ is as in \eqref{G(M)}.
\end{theorem}

\begin{proof}
Fix a partition $(A_{1},\ldots,A_{K})$ of $\Y$. Then by Proposition \ref{K-CIR implications}
\begin{align*}%\label{}
\psi_{t}\Big(&\Gamma^{\beta+\ss}_{\alpha+\sum_{i=1}^{K_{m}}m_{i}\delta_{y_{i}^{*}}}(A_{1},\ldots,A_{K})\Big)\\
=&\,\sum_{i=0}^{|\mm|}\mathrm{Bin}(|\mm|-i;\,|\mm|,p(t))\mathrm{Ga}(\na+|\mm|-i,\beta+\SS_{t})\sum_{\substack{\oo \le \ii\le \tilde\mm, |\ii|=i}}p(\ii;\, \tilde\mm,i) \pi_{\aa+\tilde\mm-\ii},
\end{align*}
where $\Gamma^{\beta+\ss}_{\alpha+\sum_{i=1}^{K_{m}}m_{i}\delta_{y_{i}^{*}}}(A_{1},\ldots,A_{K})$ denotes $\Gamma^{\beta+\ss}_{\alpha+\sum_{i=1}^{K_{m}}m_{i}\delta_{y_{i}^{*}}}(\cdot)$ evaluated on $(A_{1},\ldots,A_{K})$ and $\tilde\mm$ are the multiplicities yielded by the projection of $\mm$ onto $(A_{1},\ldots,A_{K})$.
Use now \eqref{gamma-dirichlet algebra} and \eqref{tilde pmn} to write the right hand side of \eqref{gamma propagation} as
\begin{align}\notag
\sum_{\nn\in G(\mm)}&
\tilde p_{\mm,\nn}(t)
\Gamma^{\beta+\SS_{t}}_{\alpha+\sum_{i=1}^{K_{m}}n_{i}\delta_{y_{i}^{*}}}\\
=&\,\sum_{i=0}^{|\mm|}\mathrm{Bin}(|\mm|-i;\,|\mm|,p(t))\text{Ga}(\na+|\mm|-i,\beta+\SS_{t})\sum_{\substack{\oo \le \nn\le \mm, |\nn|=i}}p(\nn;\, \mm,i)
\Pi_{\alpha+\sum_{j=1}^{K_{m}}(m_{j}-n_{j})\delta_{y_{j}^{*}}}.\notag
\end{align}
Since the inner sum is the only term which depends on multiplicities and  Since Dirichlet processes are characterised by their finite-dimensional projections, we are only left to show that
\begin{equation*}%\label{}
\sum_{\substack{\oo \le \nn\le \mm, |\nn|=i}}p(\nn;\, \mm,i)
\Pi_{\alpha+\sum_{j=1}^{K_{m}}(m_{j}-n_{j})\delta_{y_{j}^{*}}}(A_{1},\ldots,A_{K})=\sum_{\substack{\oo \le \ii\le \tilde\mm, |\ii|=i}}p(\ii;\, \tilde\mm,i) \pi_{\aa+\tilde\mm-\ii}
\end{equation*}
which, in view of \eqref{DP to dir}, holds if
\begin{equation*}%\label{}
\sum_{\substack{\oo \le \nn\le \mm: \tilde\nn=\ii}}p(\ii;\, \mm,i)
=p(\ii;\, \tilde\mm,i),
\end{equation*}
where $\tilde\nn$ denotes the projection of $\nn$ onto $(A_{1},,\ldots,A_{K})$.
This is the consistency with respect to merging of classes of the multivariate hypergeometric distribution, and so the result now follows by the same argument at the end of the proof of Theorem \ref{prop: FV propagation}.
% Given \eqref{DP to dir}, we only need to show that the mixture weights are consistent with respect to fragmentation and merging of classes, that is
%\begin{equation}\notag
%\sum_{\ii\in G(\mm):\, \tilde \ii=\nn}p_{\mm,\ii}(t)=p_{\tilde\mm,\nn}(t),
%\end{equation}
%where $\tilde\ii$ denotes the projection of $\ii$ onto $(A_{1},\ldots,A_{K})$.
%Using \eqref{transition probabilities}, the previous in turn reduces to
%\begin{equation}\label{hypergeometric marginal}
%\sum_{\ii\in G(\mm):\, \tilde \ii=\nn}p(\ii;\,\mm,m-i)=p(\nn;\,\tilde\mm,m-n),
%\end{equation}
%which holds by the marginalization properties of the multivariate hypergeometric distribution. Cf.~\cite{JKB97}, equation 39.3.\qed\\
%
%Proceeding now as in
%the proof  of Theorem \ref{prop: FV propagation}, it is easy to show
%that the projection of the previous expression coincides with the
%right hand side of \eqref{multi CIR propagation}, and that this in
%turn implies \eqref{gamma propagation} by the marginalization
%conditions \eqref{hypergeometric marginal} for the multivariate
%hypergeometric probabilities. \todo{we lack detail again here}
\end{proof}

The next proposition provides the update step for mixtures of gamma priors.

\begin{proposition}\label{prop: DW update}
Let $y_{m+i}\mid z\sim z$, $i=1,\ldots,n$, with
\begin{equation*}%\label{}
z\sim \sum_{\mm \in M}w_{\mm}\Gamma^{\beta+\ss}_{\alpha+\sum_{i=1}^{K_{m}}m_{i}\delta_{y_{i}^{*}}}, \quad \sum_{\mm\in M}w_\mm=1.
\end{equation*}
Then
\begin{align}\label{eq: NP-update}
\phi_{\yy_{m+1:m+n}}\Bigg(\sum_{\mm \in M}w_{\mm}\Gamma^{\beta+\ss}_{\alpha+\sum_{i=1}^{K_{m}}m_{i}\delta_{y_{i}^{*}}}\Bigg)
=
\sum_{\nn\in t(\yy_{m+1:m+n},M)}\hat w_{\nn}
\Gamma^{\beta+\ss+n-m}_{\alpha+\sum_{i=1}^{K_{m+n}}n_{i}\delta_{y_{i}^{*}}},
\end{align}
with $t(\cdot)$ is as in Proposition \ref{algorithm infinite alleles} and $\hat w_{\nn}$ as in \eqref{updated mixture weights}.
\end{proposition}

\begin{proof}
Since $z_{\mm}:=(z\mid \mm)\sim\Gamma^{\beta+\ss}_{\alpha+\sum_{i=1}^{K_{m}}m_{i}\delta_{y_{i}^{*}}}$, from \eqref{poisson data} we have
\begin{equation}\notag
y_{m+1},\ldots,y_{n}\mid z,\mm,n\overset{iid}{\sim}z_{\mm}/|z_{\mm}|,\quad \quad
n\mid z_{\mm}\sim\text{Po}(|z_{\mm}|).
\end{equation}
Using \eqref{gamma-dirichlet algebra} we have
\begin{align*}%\label{}
\Gamma^{\beta+\ss}_{\alpha+\sum_{i=1}^{K_{m}}m_{i}\delta_{y_{i}^{*}}}
=
\text{Ga}(\theta+|\mm|,\beta+s)\Pi_{\alpha+\sum_{i=1}^{K_{m}}m_{i}\delta_{y_{i}^{*}}},
\end{align*}
that is $|z_{\mm}|$ and $z_{\mm}/|z_{\mm}|$ are independent with $\text{Ga}(\theta+|\mm|,\beta+s)$ and $\Pi_{\alpha+\sum_{i=1}^{K_{m}}m_{i}\delta_{y_{i}^{*}}}$ distribution respectively. \eqref{gamma conjugacy} now implies that
\begin{equation*}%\label{}
z_{\mm}\mid \yy_{m+1:m+n}\sim \text{Ga}(\theta+|\nn|,\beta+s+n-m)\Pi_{\alpha+\sum_{i=1}^{K_{m+n}}n_{i}\delta_{y_{i}^{*}}}=
\Gamma^{\beta+\ss+n-m}_{\alpha+\sum_{i=1}^{K_{m+n}}n_{i}\delta_{y_{i}^{*}}}
\end{equation*}
where $\nn$ are the multiplicities of the distinct values in $\yy_{1:n}$.
Finally, by the independence of $|z_{\mm}|$ and $z_{\mm}/|z_{\mm}|$, the conditional distribution of the mixing measure follows by the same argument used in Proposition \ref{prop: FV update multi}.
\end{proof}

The successive iteration of the update and prediction step given by Theorem \ref{prop: DW propagation} and Proposition \ref{prop: DW update} provides the filter for DW signals with data as in \eqref{poisson data}. This is summarised in the next Proposition.

\begin{proposition}\label{algorithm DW}
Consider the family of finite mixtures of gamma random measures
\begin{equation*}%\label{}
\Fgamma =
\Bigg\{\sum_{\mm\in M}w_{\mm}\Gamma^{\beta+\ss}_{\alpha+\sum_{i=1}^{K_{m}}m_{i}\delta_{y_{i}^{*}}}:\, s>0,\, M \subset \M,\ |M|<\infty,\, w_{\mm}\ge0,\  \sum_{\mm\in M}w_\mm=1\Bigg\},
\end{equation*}
with $\M$ as in \eqref{set of multiplicities}.
Then, when $Z_{t}$ has generator \eqref{DW-generator} and data are as in \eqref{poisson data}, $\Fgamma$ is closed under the application of the update and prediction operators \eqref{update operator} and \eqref{prediction operator}. Specifically,
\begin{equation}\label{DW algorithm update}
\phi_{\yy_{m+1:m+n}}\Bigg(\sum_{\mm \in M}w_{\mm}\Gamma^{\beta+\ss}_{\alpha+\sum_{i=1}^{K_{m}}m_{i}\delta_{y_{i}^{*}}}\Bigg)
=
\sum_{\nn\in t(\yy_{m+1:m+n},M)}\hat w_{\nn}
\Gamma^{\beta+\ss+n-m}_{\alpha+\sum_{i=1}^{K_{m+n}}n_{i}\delta_{y_{i}^{*}}},
\end{equation}
with
\begin{equation}
  \label{eq:update-mix FV}\notag
  \begin{split}
t(\yy,\Lambda) :=&\, \{\nn: \nn=t(\yy,\mm), \mm \in \Lambda \} \\
\hat w_{\nn}\propto&\,
w_{\mm}\,\mathrm{PU}_{\alpha}(\yy_{m+1:m+n}\mid \yy_{1:m})
   \quad   \textrm{for}\quad  \nn = t(\yy,\mm) \,,\sum_{\nn
      \in t(\yy,\Lambda)} \hw_\nn =1\,,
  \end{split}
\end{equation}
and
\begin{equation*}\label{DW algorithm prediction}
\psi_{t}\bigg(\sum_{\mm\in M}w_{\mm}\Gamma^{\beta+\ss}_{\alpha+\sum_{i=1}^{K_{m}}m_{i}\delta_{y_{i}^{*}}}\bigg)=
\sum_{\nn\in G(M)}
\bigg(\sum_{\mm\in M,\, \mm\ge \nn}w_{\mm}\tilde p_{\mm,\nn}(t)\bigg)
\Gamma^{\beta+\SS_{t}}_{\alpha+\sum_{i=1}^{K_{m}}n_{i}\delta_{y_{i}^{*}}}.
\end{equation*}
\end{proposition}

\begin{proof}
The update operation \eqref{DW algorithm update} follows directly from Lemma \ref{prop: FV update multi}.
The prediction operation \eqref{eq: FV algorithm propagation} for elements of $\F_{\Pi}$ follows from Theorem \ref{prop: DW propagation} together with \eqref{prediction linearity} and a rearrangement of the sums, so that
\begin{align*}%\label{prediction sum change}
\psi_{t}\Bigg(\sum_{\mm\in M}w_{\mm}\Gamma^{\beta+\ss}_{\alpha+\sum_{i=1}^{K_{m}}m_{i}\delta_{y_{i}^{*}}}\Bigg)
=&\,\sum_{\mm\in M}w_{\mm}
\sum_{\substack{\nn\in G(\mm)} }
p_{\mm,\nn }(t)
\Gamma^{\beta+\SS_{t}}_{\alpha+\sum_{i=1}^{K_{m}}m_{i}\delta_{y_{i}^{*}}}\\
=&\,\sum_{\nn\in G(M)}
\Bigg(\sum_{\mm\in M,\,\mm\ge \nn}w_{\mm}p_{\mm,\nn}(t)\Bigg)
\Gamma^{\beta+\SS_{t}}_{\alpha+\sum_{i=1}^{K_{m}}m_{i}\delta_{y_{i}^{*}}}.
\end{align*}
\end{proof}

\section{Interpretations and concluding remarks}\label{sec: discussion}

We have derived explicit filters that allow sequential evaluation of the marginal posterior distributions of hidden FV and DW signals given appropriate observations. The algorithms provided in Propositions \ref{algorithm FV} and \ref{algorithm DW}, which extend previous results on conjugacy of mixtures of Dirichlet and gamma priors \citep{A74,L82}, can be interpreted as follows. Without observations, our knowledge of the signal amounts to the prior distribution, i.e., the law of the FV or DW process, which provides the instant-wise information encoded in the stationary distributions $\Pi_{\alpha}$ and $\Gamma^{\beta}_{\alpha}$ respectively. When new observations become available, this new information is integrated into our current knowledge by means of the update operator, which modifies the prior parameters according to the atoms observed in the sample, yielding $\beta+m$ for the DW case and $\alpha+\sum_{i=1}^{K_{m}}m_{i}\delta_{y_{i}^{*}}$ for both. In the following interval of time, before furthers observations become available, the distribution of the signals becomes a mixture that gradually approaches the ergodic distribution. This implies that the signal progressively forgets the previously acquired information, which becomes obsolete and gradually swamped by the prior. This mechanisms is enforced by a time-continuous modification of the mixture weights, which modulate the sequential random removal of the atoms previously added to the base measure, governed by the death process. In addition, in the DW case a deterministic process governs the sample size parameter, which increases by jumps in the update step but decreases continuously in the propagation, gradually bringing it back to its ergodic state.
The dual process, which is related to the time reversal structure of the signal, is thus able to isolate the prior knowledge on the signal from the posterior information acquired with the observations,  and dictates how this information is to be dealt with for filtering.

For what concerns the strategy followed for proving the propagation result in Theorems \ref{prop: FV propagation} and \ref{prop: DW propagation}, one could be tempted to work directly
%the mutual singularity of Dirichlet priors with different nonatomic parameter measures (see Section 3.2 of \citealp{GR03}). This property makes the strategy used for the WF case unfeasible, since there is no likelihood to recombine with the distribution of the signal.
%It is worth noting at this point that trying to solve the prediction problem by working directly
with the duals of the FV and DW processes (see \citealp{DH82,EK93,E00}). However, this is not optimal, due to the high degree of generality of such dual processes. The simplest path for deriving the propagation step for the nonparametric signals appears to be resorting to the corresponding parametric dual by means of projections and by exploiting the filtering results for those cases.

\section*{Ackowledgements}%\label{}

The first author acknowledges financial support from the Spanish government  via grant MTM2012-37195.
The second author is supported by the European Research Council (ERC) through StG ``N-BNP'' 306406.

%%%%%%%%%%%%%%%%%%%%%%%%%%%%%%%

\section*{Appendix}\label{sec: proofs}
\phantomsection\addcontentsline{toc}{section}{Appendix}

\renewcommand{\thesection}{\Alph{section}}
\setcounter{section}{1}
\newtheorem{applemma}{Lemma}
\renewcommand{\theapplemma}{\thesection.\arabic{applemma}}
\setcounter{applemma}{0}

The following Lemma provides the transition probabilities for the death process of Theorem \ref{prop: PR14-WF-infty}.

\begin{applemma}\label{lemma: death transitions probabilities}
Let $M_{t}\subset\Z^{\infty}$ be a death process that starts from $M_{0}=\mm_{0}\in\M$ (see \eqref{set of multiplicities}) and jumps from $\mm$ to $\mm-\ee_{i}$ at rate $m_{i}(\theta+\mm-1)/2$, with generator
\begin{equation*}%\label{}
\frac{\theta+|\mm|-1}{2}\sum_{i\ge1}m_{i}h(\xx,\mm-\ee_{i})-\frac{|\mm|(\theta+|\mm|-1)}{2}h(\xx,\mm).
\end{equation*}
Then the transition probabilities for $M_{t}$ are
\begin{equation*}%\label{}
\begin{split}
p_{\mm,\mm}(t)=&\,e^{-\lambda_{|\mm|}},\\
p_{\mm,\mm-\ii }(t)=&\,
C_{|\mm| ,|\mm| -|\ii|}(t) \Bigg(\prod_{h=0}^{|\ii|-1}\lambda_{|\mm| -h}\Bigg)p(\ii;\,\mm,|\ii|),
\quad  \oo< \ii\le \mm,\\
C_{|\mm| ,|\mm| -|\ii|}(t)=&\,
(-1)^{|\ii|}\sum_{k=0}^{|\ii|}\frac{e^{-\lambda_{|\mm| -k}t}}{\prod_{0\le h\le |\ii|,h\ne k}(\lambda_{|\mm| -k}-\lambda_{|\mm| -h})},
\end{split}
\end{equation*}
and 0 otherwise.
\end{applemma}
\begin{proof}
Since $|\mm_{0}|<\infty$, for any such $\mm_{0}$ the proof is analogous to that of Proposition 2.1 in \cite{PR14}.
\end{proof}

%The next Lemma provides an alternative way of writing the propagation step for FV given in Theorem \ref{prop: FV propagation}.
%
%\begin{applemma}\label{comparison}
%Let $X_{t}$ be as in Theorem \ref{prop: FV propagation}. Then
%\begin{align*}
%\psi_{t}\Bigg(\Pi_{\alpha+\sum_{i=1}^{K_{m}}m_{i}\delta_{y_{i}^{*}}}\Bigg)
%=&\,\sum_{i=0}^{|\mm|}C_{|\mm| ,|\ii|}(t) \Bigg(\prod_{h=0}^{|\mm| -i-1}\lambda_{|\mm| -h}\Bigg)\sum_{\substack{\oo \le \ii\le \mm, |\ii|=i}}p(\ii;\,\mm,i)
%\Pi_{\alpha+\sum_{j=1}^{K_{m}}i_{j}\delta_{y_{j}^{*}}}.
%\end{align*}
%\end{applemma}
%\begin{proof}
%The right hand side follows immediately by substituting the expression for $p_{\mm,\nn}(t)$ into the right hand side of \eqref{FV propagation} and by factorising the sum over $\oo\le \ii\le \mm$ into a sum over $i=0,\ldots,|\mm|$ and one over vectors $\oo\le \ii\le \mm$ such that $|\ii|=i$.
%\end{proof}

The following Lemma recalls the propagation step for one dimensional CIR processes.

\begin{applemma}\label{lemma: 1-CIR propagation}
Let $Z_{i,t}$ be a CIR process with generator \eqref{CIR generator} and invariant distribution $\mathrm{Ga}(\alpha_{i},\beta)$. Then
\begin{equation}\label{CIR propagation}\notag
\psi_{t}\big(\mathrm{Ga}(\alpha_{i}+m,\beta+\ss)\big)
=\sum\nolimits_{j=0}^{m}\mathrm{Bin}(m-j;\,m,p(t))
\mathrm{Ga}(\alpha_{i}+m-j,\beta+\SS_{t}),
\end{equation}
where
\begin{equation}\notag
p(t)=\frac{\beta}{(\beta+\ss)e^{\beta t/2}-\ss},\quad
%\SS_{t}=\frac{\beta e^{-\beta t/2}}{1+\beta/\ss-e^{-\beta t/2}}
\SS_{t}=\frac{\beta s}{(s+\beta)e^{\beta t/2}-s}, \quad \SS_{0}=\ss.
\end{equation}
\end{applemma}

\begin{proof}
It follows from Section 3.1 in \cite{PR14} by letting $\alpha=\delta/2$,  $\beta=\gamma/\sigma^{2}$ and $\SS_{t}=\Theta_{t}-\beta$.
\end{proof}

%%%%%%%%%%%%%%%%%%%%%%%%%%%%%%%%

\end{document}